\documentclass[reqno]{amsart}

\usepackage{amssymb}
\usepackage{graphicx}
\usepackage{amscd}
\usepackage[pagebackref]{hyperref}
\usepackage{color}
\usepackage{tabularx}
\usepackage[table]{xcolor}
\usepackage{float}
\usepackage{graphics,amsmath,amssymb}
\usepackage{amsthm}
\usepackage{amsfonts}
\usepackage{latexsym}
\usepackage{epsf}
\usepackage{xifthen}
\usepackage{mathrsfs}
\usepackage{dsfont}
\usepackage{makecell}
\usepackage{subfig}
\usepackage{amsmath}
\allowdisplaybreaks[4]
\usepackage{listings}
\usepackage{etoolbox}
\usepackage{fancyhdr}
\usepackage{pdflscape}
\usepackage[title,toc,titletoc]{appendix}
\usepackage{enumitem}
\usepackage[noadjust]{cite}
\usepackage{tikz}
\usetikzlibrary{automata,positioning,arrows}
\usepackage{young}
\usepackage[object=vectorian]{pgfornament} %%  http://altermundus.com/pages/tkz/ornament/index.html
\usepackage{lipsum,tikz}
\usepackage{multirow}

\hypersetup{
	colorlinks=true, %set true if you want colored links
	linktoc=all, %set to all if you want both sections and subsections linked
	linkcolor=blue} %choose some color if you want links to stand out

\numberwithin{equation}{section}

\theoremstyle{theorem}
\newtheorem{theorem}{Theorem}[section]
\newtheorem*{theorem*}{Theorem}

\newtheorem{lemma}[theorem]{Lemma}

\newtheorem{innercustomgeneric}{\customgenericname}
\providecommand{\customgenericname}{}
\newcommand{\newcustomtheorem}[2]{%
	\newenvironment{#1}[1]
	{%
		\renewcommand\customgenericname{#2}%
		\renewcommand\theinnercustomgeneric{##1}%
		\innercustomgeneric
	}
	{\endinnercustomgeneric}
}
\newcustomtheorem{ctheorem}{Theorem}
\newcustomtheorem{clemma}{Lemma}

\theoremstyle{definition}
\newtheorem{definition}[theorem]{Definition}

\newtheorem*{example*}{Example}
\newtheorem*{examples*}{Examples}

\newtheorem*{remark*}{Remark}
\newtheorem*{remarks*}{Remarks}
\newtheorem*{note*}{Note}

\newtheoremstyle{named}{}{}{\itshape}{}{\bfseries}{.}{.5em}{#1\thmnote{ #3}}
\theoremstyle{named}

\DeclareSymbolFont{mathbbb}{U}{bbold}{m}{n}
\DeclareMathSymbol{\ph}{\mathalpha}{mathbbb}{'010}
\DeclareMathSymbol{\ps}{\mathalpha}{mathbbb}{'011}

\newcommand{\cS}{\mathcal{S}}
\newcommand{\cT}{\mathcal{T}}

\newcommand{\cY}{\mathcal{Y}}

\newcommand{\hxi}{\hat{\xi}}
\newcommand{\hzeta}{\hat{\zeta}}
\newcommand{\hgamma}{\hat{\gamma}}
\newcommand{\htt}{\hat{t}}
\newcommand{\hL}{\hat{L}}
\newcommand{\hN}{\hat{N}}
\newcommand{\hS}{\hat{S}}
\newcommand{\hT}{\hat{T}}
\newcommand{\hV}{\hat{V}}

\newcommand{\hZ}{\hat{Z}}
\newcommand{\hcS}{\hat{\mathcal{S}}}
\newcommand{\hcT}{\hat{\mathcal{T}}}

\newcommand{\hcY}{\hat{\mathcal{Y}}}

\title[Congruence multiplicities between Ramanujan's theta functions]{On the occurrence of congruence multiplicities between Ramanujan's theta functions}

\author[S. Chern]{Shane Chern}
\address[S. Chern]{Fakult\"at f\"ur Mathematik, Universit\"at Wien, Oskar-Morgenstern-Platz 1, Wien 1090, Austria}
\email{chenxiaohang92@gmail.com, xiaohangc92@univie.ac.at}

\author[N. A. Smoot]{Nicolas Allen Smoot}
\address[N. A. Smoot]{Fakult\"at f\"ur Mathematik, Universit\"at Wien, Oskar-Morgenstern-Platz 1, Wien 1090, Austria}
\email{nicolas.allen.smoot@univie.ac.at}

\author[D. Tang]{Dazhao Tang}
\address[D. Tang]{School of Mathematical Sciences, Chongqing Normal University, Chongqing 401331, P.R. China}
\email{dazhaotang@sina.com}

\date{}



\begin{document}
	
\sloppy

\begin{abstract}
Recently, the first and third authors initiated a study of arithmetical relationships between Ramanujan's theta functions $\varphi(-q)$ and $\psi(q)$. In this work, we prove eight families of internal congruences modulo arbitrary powers of $3$ and $5$ for infinite series related to the two theta functions. We also show that there exist isomorphisms between the congruence families for $\varphi(-q)$ and $\psi(q)$, which can be realized by the study of congruence multiplicities previously studied by Garvan, Sellers, and the second author. We believe that such equivalences are exclusive, at least on the congruence subgroups $\Gamma_0(6)$ and $\Gamma_0(10)$. In the end, we show how a simple manipulation of function field extensions allows us to predict whether additional isomorphisms to our congruences occur.
\end{abstract}

\maketitle

\section{Introduction}

It is well-known that the coefficients of modular forms often exhibit striking divisibility properties with respect to prime powers in fixed linear progressions.  These divisibility properties vary enormously with respect to their overall difficulty: some of them can be proved by relatively elementary methods, while others remain standing conjectures.

In their recent work \cite{CT2025}, Chern and Tang initiated a study of internal congruences modulo prime powers for infinite series arising from two specializations of \emph{Ramanujan's theta functions} \cite[p.~36, Entry~22]{Ber1991}:
\begin{align*}
	\varphi(-q):=\prod_{k\ge 1} \frac{(1-q^{k})^2}{1-q^{2k}},\qquad\qquad
	\psi(q):=\prod_{k\ge 1} \frac{(1-q^{2k})^2}{1-q^{k}}.
\end{align*}

Given a series of the form $\sum_{n\ge 0} c(n)q^n$, we define an \emph{internal} congruence, as a set of relations for all $n\ge 0$:
\begin{align*}
	c(An+B) \equiv c(A'n+B')\pmod{M},
\end{align*}
where the modulus $M$ and paratemers $0\le B < A$ and $0\le B'<A'$ are fixed. The \emph{differences} of the coefficients in their associated progressions are understood to vanish modulo $M$, even if the coefficients themselves generally do not.

Considering the series
\begin{align*}
	\sum_{n\ge 0}\ph_3(n)q^n :=\dfrac{\varphi(-q^3)}{\varphi(-q)},\qquad\qquad \sum_{n\ge 0} \ps_3(n) q^n:= \frac{\psi(q^3)}{\psi(q)},
\end{align*}
what was shown in \cite[p.~3, Theorems~1.1 and 1.2]{CT2025} reads as follows:

\begin{theorem}[Chern--Tang]
	For any $\alpha\ge 1$ and $n\ge 0$,
	\begin{align}\label{eq:ph3}
		\ph_3\big(3^{2\alpha-1}n\big)\equiv \ph_3\big(3^{2\alpha+1}n\big)\pmod{3^{\alpha+2}},
	\end{align}
	and
	\begin{align}\label{eq:ps3}
		\ps_3{\left(3^{2\alpha-1}n+\frac{3^{2\alpha}-1}{4}\right)}\equiv 
		\ps_3{\left(3^{2\alpha+1}n+\frac{3^{2\alpha+2}-1}{4}\right)}\pmod{3^{\alpha+2}}.
	\end{align}
\end{theorem}

Almost at the same time, Garvan, Sellers and Smoot \cite{GSS2024} discovered a means of building isomorphisms between congruence families for different modular forms. In particular, the equivalence between the congruences \eqref{eq:ph3} and \eqref{eq:ps3} was examined in \cite[Section~4.3]{GSS2024}.

This theory has subsequently been studied in \cite{SellersS0}, wherein the phenomenon of such isomorphisms is referred to as \emph{modular congruence multiplicity}.  It appears that such multiplicities are a prevalent phenomenon: arithmetical properties of interest can manifest themselves in the coefficients of multiple different modular forms.  Moreover, this phenomenon tends to only be possible with congruences of a more difficult kind \cite[Section~4.4]{GSS2024}, generally associated with a composite level, or a more complex topology of the underlying modular curve.  This behavior appears impossible with the easier congruences.

In this work, we shall first witness that internal congruences like \eqref{eq:ph3} and \eqref{eq:ps3} are \emph{not} rare for series involving the two theta functions. Let
\begin{align*}
	\Phi = \Phi(\tau) = \sum_{n\ge 0} \ph(n) q^n := \frac{1}{\varphi(-q)},
\end{align*}
where we adopt the convention that $q:=e^{2\pi i \tau}$ with $\tau\in \mathbb{H}$, the upper half complex plane. In the meantime, for positive integers $k$, we define
\begin{align*}
	\Phi_k = \Phi_{k}(\tau) = \sum_{n\ge 0} \ph_{k}(n) q^n := \frac{\varphi(-q^k)}{\varphi(-q)}.
\end{align*}

Then there are two analogs to \eqref{eq:ph3}.

\begin{theorem}\label{th:ph-mod3}
	For any $\alpha\ge 1$ and $n\ge 0$,
	\begin{align}
		\ph \big(3^{2\alpha+1} n\big) &\equiv \ph \big(3^{2\alpha-1} n\big) \pmod{3^{3\alpha-2}},\label{eq:ph-mod3-odd}\\
		\ph \big(3^{2\alpha+2} n\big) &\equiv \ph \big(3^{2\alpha} n\big) \pmod{3^{3\alpha}}.\label{eq:ph-mod3-even}
	\end{align}
\end{theorem}

\begin{theorem}\label{th:ph9-mod3}
	For any $\alpha\ge 1$ and $n\ge 0$,
	\begin{align}\label{eq:ph9-mod3}
		\ph_9 \big(3^{\alpha+1} n\big) \equiv \ph_9 \big(3^{\alpha} n\big) \pmod{3^{2\alpha}}.
	\end{align}
\end{theorem}

Meanwhile, if we replace the modulus from powers of $3$ to powers of $5$, internal congruences of a like nature are retained.

\begin{theorem}\label{th:ph-mod5}
	For any $\alpha\ge 1$ and $n\ge 0$,
	\begin{align}\label{eq:ph-mod5}
		\ph \big(5^{2\alpha+2} n\big) &\equiv \ph \big(5^{2\alpha} n\big) \pmod{5^{\alpha}}.
	\end{align}
\end{theorem}

\begin{theorem}\label{th:ph25-mod5}
	For any $\alpha\ge 1$ and $n\ge 0$,
	\begin{align}\label{eq:ph25-mod5}
		\ph_{25} \big(5^{\alpha+1} n\big) \equiv \ph_{25} \big(5^{\alpha} n\big) \pmod{5^{\alpha}}.
	\end{align}
\end{theorem}

Next, putting the above families of internal congruences into the context of congruence multiplicities, we may accordingly work out, as treated in \cite{GSS2024}, the isomorphic congruence families associated with the theta function $\psi(q)$ via the application of a certain Atkin--Lehner involution. Let us define
\begin{align*}
	\Psi = \Psi(\tau) = \sum_{n\ge 0} \ps(n) q^n := \frac{1}{\psi(q)},
\end{align*}
and for positive integers $k$,
\begin{align*}
	\Psi_k = \Psi_{k}(\tau) = \sum_{n\ge 0} \ps_{k}(n) q^n := \frac{\psi(q^k)}{\psi(q)}.
\end{align*}

The following two families of internal congruences are equivalent to those in Theorems~\ref{th:ph-mod3} and \ref{th:ph9-mod3}.

\begin{theorem}\label{th:ps-mod3}
	For any $\alpha\ge 1$ and $n\ge 0$,
	\begin{align}
		\ps \!\left(3^{2\alpha+1} n + \frac{5\cdot 3^{2\alpha+1}+1}{8}\right) &\equiv \ps \!\left(3^{2\alpha-1} n + \frac{5\cdot 3^{2\alpha-1}+1}{8}\right) \pmod{3^{3\alpha-2}},\label{eq:ps-mod3-odd}\\
		\ps \!\left(3^{2\alpha+2} n + \frac{7\cdot 3^{2\alpha+2}+1}{8}\right) &\equiv \ps \!\left(3^{2\alpha} n + \frac{7\cdot 3^{2\alpha}+1}{8}\right) \pmod{3^{3\alpha}}.\label{eq:ps-mod3-even}
	\end{align}
\end{theorem}

\begin{theorem}\label{th:ps9-mod3}
	For any $\alpha\ge 1$ and $n\ge 0$,
	\begin{align}\label{eq:ps9-mod3}
		\ps_9 \big(3^{\alpha+1} n + 3^{\alpha+1} -1\big) \equiv \ps_9 \big(3^{\alpha} n + 3^{\alpha} -1\big) \pmod{3^{2\alpha}}.
	\end{align}
\end{theorem}

Also, we have isomorphisms between congruences among Theorems~\ref{th:ph-mod5} and \ref{th:ph25-mod5} and the following results.

\begin{theorem}\label{th:ps-mod5}
	For any $\alpha\ge 1$ and $n\ge 0$,
	\begin{align}\label{eq:ps-mod5}
		\ps \!\left(5^{2\alpha+2} n + \frac{7\cdot 5^{2\alpha+2}+1}{8}\right) &\equiv \ps \!\left(5^{2\alpha} n + \frac{7\cdot 5^{2\alpha}+1}{8}\right) \pmod{5^{\alpha}}.
	\end{align}
\end{theorem}

\begin{theorem}\label{th:ps25-mod5}
	For any $\alpha\ge 1$ and $n\ge 0$,
	\begin{align}\label{eq:ps25-mod5}
		\ps_{25} \big(5^{\alpha+1} n + 5^{\alpha+1} -3\big) &\equiv \ps_{25} \big(5^{\alpha} n + 5^{\alpha} -3\big) \pmod{5^{\alpha}}.
	\end{align}
\end{theorem}

Finally, it is noteworthy to remark that Bharadwaj, Hemanthkumar and Naika conjectured \eqref{eq:ph9-mod3} in \cite[p.~129, Conjecture~4.3]{BHN2018} and \eqref{eq:ps9-mod3} in \cite[p.~810, Conjecture~5.1]{HBN2019}, while for both conjectures only the starting cases with $\alpha\in\{1,2,3\}$ have been shown. In addition, \eqref{eq:ps-mod5} was first conjectured by Tang in \cite[p.~73, Conjecture~1.3]{Tang2025} with the first $20$ cases of $\alpha$ verified. These conjectures become one of the motivations for the present work.

\subsection{Outline of the proofs}

There are two keys in our analysis. In what follows, we shall use \eqref{eq:ph9-mod3} and \eqref{eq:ps9-mod3} as an illustration.  We use the notation of \cite{GSS2024} and \cite{Smoot2} for our discussion of modular curves and cusps.  See also \cite[Chapters~2--3]{Diamond} for the standard theoretical treatment.

We define an action of $\operatorname{SL}_2(\mathbb{Z})$ onto elements of $\mathbb{H}$ by
\begin{align}
\kappa\tau := \frac{a\tau+b}{c\tau+d}.\label{sl2action}
\end{align}  For a given positive integer $N$, define the congruence subgroup
\begin{align*}
\Gamma_0(N) := \left\{ \begin{pmatrix} a & b \\ c & d \end{pmatrix}\in\operatorname{SL}_2(\mathbb{Z}) : N\mid c \right\}.
\end{align*}  We will work with functions which are modular with respect to the action of (\ref{sl2action}) restricted to $\Gamma_0(N)$, especially for $N=6$ and $10$.

Such functions are equivalent to meromorphic functions on the classical modular curve $\mathrm{X}_0(N)$.  This curve is a compact Riemann surface, and its meromorphic functions are thus subject to various powerful restrictions that we can take advantage of.

For $N=6$ or $10$, the curve $\mathrm{X}_0(N)$ has genus 0, and thus admits \textit{Hauptmoduln}, i.e., functions holomorphic across the entire surface, save for a simple pole residing at a given cusp.

We introduce a \emph{Hauptmodul} corresponding to the classical modular curve $\mathrm{X}_0(6)$ at the cusp represented by $[0]_6$:
\begin{align*}
	\xi = \xi(\tau) := \left(\frac{\varphi(-q^3)}{\varphi(-q)}\right)^4.
\end{align*}
We will find that, for every $\alpha\ge 1$, the series $\sum_{n\ge 0} \ph_9 \big(3^{\alpha} n\big) q^n$ can be expressed as a polynomial in $\xi$. Hence, we define
\begin{align}\label{eq:L9-def}
	L_9(\xi,\alpha) := \sum_{n\ge 0} \ph_9 \big(3^{\alpha} n\big) q^n
\end{align}
with
\begin{align*}
	L_9(\xi,\alpha) \in \mathbb{Z}[\xi].
\end{align*}
Now, to show \eqref{eq:ph9-mod3}, it suffices to prove
\begin{align*}
	L_9(\xi,\alpha+1)- L_9(\xi,\alpha) \equiv 0 \pmod{3^{2\alpha}}.
\end{align*}

Intuitively, we would like to prove the above congruence by induction on $\alpha$ because by \eqref{eq:L9-def} we have
\begin{align}\label{eq:L9-U3}
	L_9(\xi,\alpha+2)- L_9(\xi,\alpha+1) = U_3\big(L_9(\xi,\alpha+1)- L_9(\xi,\alpha)\big),
\end{align}
where $U_k$ is the \emph{unitizing operator of degree $k$}, defined by
\begin{align*}
	U_k\left(\sum_{n} c(n) q^n\right) := \sum_{n} c(kn) q^n.
\end{align*}
Noting that $L_9(\xi,\alpha+1)- L_9(\xi,\alpha)\in \mathbb{Z}[\xi]$ as argued, all we need is the fact that for every $i\ge 0$, the series $U_3(\xi^i)$ can be represented in the polynomial ring $\mathbb{Z}[\xi]$. However, with these polynomial representations substituted into the right-hand side of \eqref{eq:L9-U3}, we find the inductive step challenging due to the chaotic $3$-adic behavior of a few initial coefficients in the polynomial expansion.

Now our \emph{first key} is that, instead of working on $L_9(\xi,\alpha+1)- L_9(\xi,\alpha)$ directly, we need to observe the \emph{factorization} that
\begin{align}\label{eq:L9-diff-factor}
	L_9(\xi,\alpha+1)- L_9(\xi,\alpha) = (1-\xi)\cdot \mathbb{Z}[\xi],
\end{align}
as will be shown in Lemma~\ref{le:L9-diff-factorization}. Heuristically, this factorization should be true because
\begin{align*}
	L_9(\xi,\alpha+1)- L_9(\xi,\alpha) = \sum_{n\ge 0} \left(\ph_9 \big(3^{\alpha+1} n\big) - \ph_9 \big(3^{\alpha} n\big)\right) q^n.
\end{align*}
At the cusp $i\infty$ so that $q=0$, we have the vanishing of the right-hand side of the above. In the meantime, $\xi(i\infty) = 1$. Hence, in the polynomial ring $\mathbb{Z}[\xi]$, it should be true that $1-\xi$ divides $L_9(\xi,\alpha+1)- L_9(\xi,\alpha)$. Once we have established the factorization \eqref{eq:L9-diff-factor}, it is easy to examine that the $3$-adic analysis for
\begin{align*}
	\frac{L_9(\xi,\alpha+1)- L_9(\xi,\alpha)}{1-\xi} \in \mathbb{Z}[\xi]
\end{align*}
can be handled more smoothly, and hence the induction can be executed with no difficulty. We refer the reader to Section~\ref{sec:ph9-mod3} for details.

As soon as \eqref{eq:ph9-mod3} has been proved, we would like to show the isomorphic congruence family \eqref{eq:ps9-mod3}. At this point, the \emph{second key} in our analysis is the Atkin--Lehner involution realized by the action of the matrix
\begin{align*}
	V := \begin{pmatrix} 2 & -1 \\ 54 & -26 \end{pmatrix}.
\end{align*}
Now we introduce a second Hauptmodul corresponding to $\mathrm{X}_0(6)$, this time living at the cusp represented by $[1/2]_6$:
\begin{align*}
	\zeta = \zeta(\tau) := q\left(\frac{\psi(q^3)}{\psi(q)}\right)^4.
\end{align*}
Two crucial facts about the above matrix action are that
\begin{align*}
	\xi(V\tau) = \zeta
\end{align*}
and
\begin{align*}
	\Phi_9(V\tau) = q\Psi_9.
\end{align*}
By lifting the polynomials $L_9(\xi,\alpha)$ to $L_9(x,\alpha)$ where $\xi$ is replaced with an indeterminate $x$, we then conclude that
\begin{align*}
	\sum_{n\ge 0} \ps_9 \big(3^{\alpha} n + 3^{\alpha} -1\big) q^n = L_9(\zeta,\alpha).
\end{align*}
Therefore, all that is left to do is to take advantage of the $3$-adic analysis for the polynomials $L_9(x,\alpha)$ as discussed earlier. The concrete arguments will be presented in Section~\ref{sec:ps9-mod3}.

In Section \ref{algconsiderations}, we will examine these congruence multiplicities in the context of function field extensions.  This allows us to make some predictions about the occurrence (or non-occurrence) of other equivalent modular congruences.  We give only an extremely simple development of the associated theory, which at any rate is all that is needed for the congruences discussed here.  Nevertheless, this approach may have substantial utility in future work, especially in consideration of the most difficult standing problems in the subject.

\section{Powers of $3$}\label{sec:mod3}

In this section, we prove Theorems~\ref{th:ph-mod3} and \ref{th:ph9-mod3}. As before, we require the Hauptmodul for $\mathrm{X}_0(6)$ at $[0]_6$:
\begin{align}
\xi = \xi(q) = \xi(\tau) := \left(\frac{\varphi(-q^3)}{\varphi(-q)}\right)^4.
\end{align}
By abuse of notation, we write $\xi(q)$ and $\xi(\tau)$ interchangeably where it should be kept in mind that $q:=e^{2\pi i \tau}$ with $\tau\in \mathbb{H}$.

We begin with the initial expressions of $U_3(\xi^i)$, which can be shown easily by a cusp analysis.

\begin{lemma}\label{le:xi-IX-poly-ini}
	We have
	\begin{align*}
		U_3(\xi) &= 10 \xi-36 \xi^2+27 \xi^3,\\
		U_3(\xi^2) &= -8 \xi+306 \xi^2-2160 \xi^3+5508 \xi^4-5832 \xi^5+2187 \xi^6,\\
		U_3(\xi^3) &= \xi-360 \xi^2+10566 \xi^3-99144 \xi^4+423549 \xi^5-944784 \xi^6\\
		&\quad+1141614 \xi^7-708588 \xi^8+177147 \xi^9,\\
		U_3(\xi^4) &= 136 \xi^2 - 14688 \xi^3 + 384642 \xi^4 - 4230144 \xi^5 + 24389424 \xi^6\\
		&\quad - 82196208 \xi^7 + 170612244 \xi^8 - 
		221079456 \xi^9\\
		&\quad + 174312648 \xi^{10} - 76527504 \xi^{11} + 
		14348907 \xi^{12}.
	\end{align*}
\end{lemma}

Next, we suppose $\upsilon=\upsilon(q)\in \mathbb{C}[[q]]$ is generic and write
\begin{align*}
	\cY_i := U_3(\upsilon \xi^i)
\end{align*}
for each $i\ge 0$. We shall show the following recurrence.

\begin{theorem}\label{th:xi-IX-poly-rec}
	For any $i\ge 3$,
	\begin{align}
		\cY_i = \big(30 \xi - 108 \xi^2 + 81 \xi^3\big) \cY_{i-1} - \big(12 \xi - 9 \xi^2\big) \cY_{i-2} + \xi \cY_{i-3}.
	\end{align}
\end{theorem}

\begin{proof}
	Let $\omega:=e^{2\pi i/3}$ be a primitive cubic root of unity. For $k\in\{0,1,2\}$, we define $\xi_k:=\xi(\omega^k q)$. For each $i\ge 1$, let
	\begin{align*}
		p_i:=\sum_{k=0}^2 \xi_k^i=3U_3(\xi^i).
	\end{align*}
	By Newton's identity \cite{Mea1992}, it is known that the \emph{elementary symmetric functions} $\sigma_j$ (with $1\le j\le 3$),
	\begin{align*}
		\sigma_j:=\sum_{1\le k_1<k_2<\cdots<k_j\le 3}\xi_{k_1}\xi_{k_2}\cdots \xi_{k_j},
	\end{align*}
	satisfy the relations:
	\begin{align*}
		\sigma_1&=p_1\\
		&=30 \xi - 108 \xi^2 + 81 \xi^3,\\
		\sigma_2&=\tfrac{1}{2}(\sigma_1 p_1-p_2)\\
		&=12 \xi - 9 \xi^2,\\
		\sigma_3&=\tfrac{1}{3}(\sigma_2 p_1-\sigma_1 p_2+p_3)\\
		&=\xi.
	\end{align*}
	Since $X=\xi_0$, $\xi_1$ and $\xi_2$ are the three roots of
	\begin{align*}
		(X-\xi_0)(X-\xi_1)(X-\xi_2)
		=X^3-\sigma_1X^2+\sigma_2X-\sigma_3,
	\end{align*}
	we conclude that for $k\in\{0,1,2\}$,
	\begin{align*}
		\xi_k^3-\sigma_1\xi_k^2+\sigma_2\xi_k-\sigma_3=0,
	\end{align*}
	so that for $i\ge 3$,
	\begin{align*}
		\xi_k^i=\sigma_1\xi_k^{i-1}-\sigma_2\xi_k^{i-2}+\sigma_3 \xi_k^{i-3}.
	\end{align*}
	For $k\in\{0,1,2\}$, we further define $\upsilon_k:=\upsilon(\omega^k q)$. Then
	\begin{align*}
		\upsilon_k\xi_k^i=\sigma_1\upsilon_k\xi_k^{i-1}-\sigma_2\upsilon_k\xi_k^{i-2}+\sigma_3 \upsilon_k\xi_k^{i-3}.
	\end{align*}
	Finally, it is clear that for every $j\ge 0$,
	\begin{align*}
		\cY_j = U_3(\upsilon \xi^j) &= \tfrac{1}{3}\big(\upsilon(q)\xi(q)^j+\upsilon(\omega q)\xi(\omega q)^j+\upsilon(\omega^2 q)\xi(\omega^2 q)^j\big)\\
		&= \tfrac{1}{3}\big(\upsilon_0\xi_0^j+\upsilon_1\xi_1^j+\upsilon_2\xi_2^j\big).
	\end{align*}
	The claimed recurrence then follows by summing the previous relation over $k\in\{0,1,2\}$.
\end{proof}

By choosing $\upsilon = \xi$ and recalling Lemma~\ref{le:xi-IX-poly-ini}, it is clear that for any $i\ge 1$,
\begin{align}\label{eq:U-xi-poly}
	U_3(\xi^i) \in \xi\cdot \mathbb{Z}[\xi].
\end{align}

Toward the factorization \eqref{eq:L9-diff-factor}, an important observation revolves around
\begin{align*}
	\cT_i:=U_3\big((1-\xi)\xi^i\big),
\end{align*}
defined for $i\ge 1$.

\begin{lemma}\label{le:T-div}
	For every $i\ge 1$, we always have
	\begin{align*}
		\frac{\cT_i}{\xi(1-\xi)}\in \mathbb{Z}[\xi].
	\end{align*}
\end{lemma}

\begin{proof}
	Note that $\cT_i = U_3(\xi^{i})-U_3(\xi^{i+1})$. By Lemma~\ref{le:xi-IX-poly-ini}, we see that the claim is true for $1\le i\le 3$. The rest then follows from Theorem~\ref{th:xi-IX-poly-rec} by choosing $\upsilon = \xi(1-\xi)$.
\end{proof}

Throughout, given any integer $n$ and prime $p$, we denote by $\nu_p(n)$ the \emph{$p$-adic valuation} of $n$ and adopt the convention that $\nu_p(0)=\infty$.

In light of Lemma~\ref{le:T-div}, we write for $i\ge 1$,
\begin{align*}
	\cT_i = (1-\xi) \sum_{j\ge 1} T_i(j) \xi^j.
\end{align*}
We may bound the $3$-adic valuation of each $T_i(j)$ from below.

\begin{theorem}\label{th:nu-T}
	For any $i,j\ge 1$,
	\begin{align}\label{eq:nu-T}
		\nu_3\big(T_i(j)\big) \ge j-i+2.
	\end{align}
\end{theorem}

\begin{proof}
	By Lemma~\ref{le:xi-IX-poly-ini}, it is straightforward to verify that the claim holds for every $1\le i\le 3$ and $j\ge 1$. Meanwhile, the recurrence for $\cT_i$ is
	\begin{align*}
		\cT_i = \big(30 \xi - 108 \xi^2 + 81 \xi^3\big) \cT_{i-1} - \big(12 \xi - 9 \xi^2\big) \cT_{i-2} + \xi \cT_{i-3},
	\end{align*}
	so that
	\begin{align*}
		T_i(j) &= 30 T_{i-1}(j-1) -108 T_{i-1}(j-2) + 81T_{i-1}(j-3)\\
		&\quad -12 T_{i-2}(j-1) + 9T_{i-2}(j-2) + T_{i-3}(j-1).
	\end{align*}
	Therefore,
	\begin{align*}
		\nu_3\big(T_i(j)\big) &\ge \min\big\{ 1+\nu_3\big(T_{i-1}(j-1)\big),3+\nu_3\big(T_{i-1}(j-2)\big),\\
		&\quad 4+\nu_3\big(T_{i-1}(j-3)\big),1+\nu_3\big(T_{i-2}(j-1)\big),\\
		&\quad 2+\nu_3\big(T_{i-2}(j-2)\big),\nu_3\big(T_{i-3}(j-1)\big)\big\}.
	\end{align*}
	It turns out that a direct application of induction on $i$ gives the required claim.
\end{proof}

\subsection{Series $\Phi_9$}\label{sec:ph9-mod3}

Recall that
\begin{align*}
	\sum_{n\ge 0}\ph_9(n)q^n =\dfrac{\varphi(-q^9)}{\varphi(-q)}.
\end{align*}

\begin{lemma}\label{le:ph9-xi}
	For any $\alpha\ge 1$,
	\begin{align}\label{eq:ph9-xi}
		\sum_{n\ge 0}\ph_9\big(3^\alpha n\big)q^n = U_3^{(\alpha-1)}(\xi),
	\end{align}
	where $U_3^{(m)}$ means iterating $m$ times the operator $U_3$.
\end{lemma}

\begin{proof}
	The $\alpha=1$ case of \eqref{eq:ph9-xi} follows by a routine cusp analysis and the remaining cases hold by iteratively applying the $U_3$-operator.
\end{proof}

It is now clear from \eqref{eq:U-xi-poly} that $\sum_{n\ge 0}\ph_9\big(3^\alpha n\big)q^n$ can be expressed as a polynomial in $\mathbb{Z}[\xi]$ for every $\alpha\ge 1$. That is, as in \eqref{eq:L9-def}, we define
\begin{align}\label{eq:L9-def-1}
	L_9(\xi,\alpha) := \sum_{n\ge 0} \ph_9 \big(3^{\alpha} n\big) q^n
\end{align}
with
\begin{align*}
	L_9(\xi,\alpha) \in \mathbb{Z}[\xi].
\end{align*}
Furthermore, according to Lemmas~\ref{le:xi-IX-poly-ini} and \ref{le:ph9-xi},
\begin{align}\label{eq:ph9-i=1}
	L_9(\xi,2) - L_9(\xi,1) = U_3(\xi) - \xi = (1-\xi)(9 \xi - 27 \xi^2).
\end{align}
Noting the fact that
\begin{align*}
	L_9(\xi,\alpha+2)- L_9(\xi,\alpha+1) = U_3\big(L_9(\xi,\alpha+1)- L_9(\xi,\alpha)\big),
\end{align*}
and recalling Lemma~\ref{le:T-div}, we conclude the following relation.

\begin{lemma}\label{le:L9-diff-factorization}
	For any $\alpha\ge 1$,
	\begin{align}
		L_9(\xi,\alpha+1)- L_9(\xi,\alpha) \in \xi(1-\xi)\cdot \mathbb{Z}[\xi].
	\end{align}
\end{lemma}

Let us write
\begin{align*}
	L_9(\xi,\alpha+1)- L_9(\xi,\alpha) = (1-\xi) \sum_{j\ge 1} Z_\alpha(j) \xi^j.
\end{align*}

We are ready to prove Theorem~\ref{th:ph9-mod3}.

\begin{proof}[Proof of Theorem~\ref{th:ph9-mod3}]
	It is sufficient to show that for every $\alpha\ge 1$ and $j\ge 1$,
	\begin{align}\label{eq:Z-ineq}
		\nu_3\big(Z_\alpha(j)\big) \ge 2\alpha + j -1.
	\end{align}
	Now we apply induction on $\alpha$. Clearly, the above is true for $\alpha=1$ in light of \eqref{eq:ph9-i=1}. For $\alpha\ge 2$,
	\begin{align*}
		(1-\xi) \sum_{j\ge 1} Z_\alpha(j) \xi^j &= U_3\left((1-\xi) \sum_{j\ge 1} Z_{\alpha-1}(j) \xi^j\right)\\
		&= \sum_{j\ge 1} Z_{\alpha-1}(j) \cdot \cT_j\\
		&= \sum_{j\ge 1} Z_{\alpha-1}(j) \cdot (1-\xi) \sum_{k\ge 1} T_j(k) \xi^k.
	\end{align*}
	Hence,
	\begin{align*}
		\sum_{j\ge 1} Z_\alpha(j) \xi^j = \sum_{j\ge 1} Z_{\alpha-1}(j) \sum_{k\ge 1} T_j(k) \xi^k,
	\end{align*}
	so that
	\begin{align*}
		Z_\alpha(j) = \sum_{l\ge 1} Z_{\alpha-1}(l) T_l(j).
	\end{align*}
	This further implies that
	\begin{align*}
		\nu_3\big(Z_\alpha(j)\big) &\ge \min_{l\ge 1} \big\{\nu_3\big(Z_{\alpha-1}(l)\big)+\nu_3\big(T_l(j)\big)\big\}\\
		&\ge \min_{l\ge 1} \big\{\big(2(\alpha-1)+l-1\big)+\big(j-l+2\big)\big\}\\
		&= 2\alpha+j-1,
	\end{align*}
	where we have applied the inductive assumption for $\nu_3\big(Z_{\alpha-1}(l)\big)$ and the inequality \eqref{eq:nu-T} for $\nu_3\big(T_l(j)\big)$. The desired inequality \eqref{eq:Z-ineq} is therefore established.
\end{proof}

\subsection{Series $\Phi$}

Let us write
\begin{align*}
	\gamma:=\Phi_9=\frac{\varphi(-q^9)}{\varphi(-q)}.
\end{align*}
At this point, we require the following relations that can be established by a routine cusp analysis.

\begin{lemma}\label{le:gamma-xi-IX-poly-ini}
	We have
	\begin{align*}
		U_3(\gamma) &= \xi,\\
		U_3(\gamma\xi) &= -5 \xi + 60 \xi^2 - 135 \xi^3 + 81 \xi^4,\\
		U_3(\gamma\xi^2) &= \xi - 162 \xi^2 + 2349 \xi^3 - 10935 \xi^4 + 21870 \xi^5 - 19683 \xi^6 + 6561 \xi^7,\\
		U_3(\gamma\xi^3) &= 91 \xi^2 - 5733 \xi^3 + 90207 \xi^4 - 597051 \xi^5 + 2028078 \xi^6 - 3838185 \xi^7\\
		&\quad + 4094064 \xi^8 - 
		2302911 \xi^9 + 531441 \xi^{10},\\
		U_3(\gamma\xi^4) &= -17 \xi^2 + 4743 \xi^3 - 211599 \xi^4 + 3485187 \xi^5 - 28479114 \xi^6\\
		&\quad + 133063641 \xi^7 - 382794984 \xi^8 + 
		701679267 \xi^9 - 822139227 \xi^{10}\\
		&\quad + 596276802 \xi^{11} - 243931419 \xi^{12} + 43046721 \xi^{13}.
	\end{align*}
\end{lemma}

In light of these retaions, if we further choose $\upsilon = \gamma$ in Theorem~\ref{th:xi-IX-poly-rec}, it is true that for any $i\ge 0$,
\begin{align}\label{eq:U-gamma-xi-poly}
	U_3(\gamma\xi^i) \in \xi\cdot \mathbb{Z}[\xi].
\end{align}

Next, for $i\ge 0$, let
\begin{align*}
	\cS_i:=U_3\big(\gamma(1-\xi)\xi^i\big).
\end{align*}
We have the following analog to Lemma~\ref{le:T-div}.

\begin{lemma}\label{le:S-div}
	For every $i\ge 0$, we always have
	\begin{align*}
		\frac{\cS_i}{\xi(1-\xi)}\in \mathbb{Z}[\xi].
	\end{align*}
\end{lemma}

\begin{proof}
	Note that $\cS_i = U_3(\gamma\xi^{i})-U_3(\gamma\xi^{i+1})$. Then we only need to use Lemma~\ref{le:gamma-xi-IX-poly-ini} and Theorem~\ref{th:xi-IX-poly-rec} with $\upsilon = \gamma(1-\xi)$.
\end{proof}

In light of Lemma~\ref{le:S-div}, we write for $i\ge 0$,
\begin{align*}
	\cS_i = (1-\xi) \sum_{j\ge 1} S_i(j) \xi^j.
\end{align*}

\begin{theorem}\label{th:nu-S}
	For any $i,j\ge 1$,
	\begin{align}\label{eq:nu-S}
		\nu_3\big(S_i(j)\big) \ge j-i+1.
	\end{align}
\end{theorem}

\begin{proof}
	The initial cases for $1\le i\le 3$ follow from Lemma~\ref{le:gamma-xi-IX-poly-ini} and the inductive argument is identical to that for Theorem~\ref{th:nu-T}.
\end{proof}

Recall that
\begin{align*}
	\sum_{n\ge 0}\ph(n)q^n =\dfrac{1}{\varphi(-q)}.
\end{align*}

For each $\alpha\ge 1$, we define
\begin{align*}
	L_\alpha := \begin{cases}
		\displaystyle\varphi(-q^3)\sum_{n\ge 0} \ph \big(3^{\alpha} n\big) q^n, & \text{if $\alpha$ is odd},\\[20pt]
		\displaystyle\varphi(-q)\sum_{n\ge 0} \ph \big(3^{\alpha} n\big) q^n, & \text{if $\alpha$ is even}.
	\end{cases}
\end{align*}
It is clear that for every $\alpha\ge 1$,
\begin{align}
	L_{2\alpha} &= U_3(L_{2\alpha-1}),\label{eq:L-U3-even}\\
	L_{2\alpha+1} &= U_3(\gamma L_{2\alpha}).\label{eq:L-U3-odd}
\end{align}
We also have initial expressions:
\begin{align*}
	L_1 &= U_3(\gamma)\\
	&= \xi,\\
	L_2 &= U_3(L_1) = U_3(\xi)\\
	&= 10 \xi-36 \xi^2+27 \xi^3,\\
	L_3 &= U_3(\gamma L_2) = U_3\big(\gamma(10 \xi-36 \xi^2+27 \xi^3)\big)\\
	&= -86 \xi + 8889 \xi^2 - 240705 \xi^3 + 2830059 \xi^4\\
	&\quad - 16907697 \xi^5 + 55466694 \xi^6 - 103867191 \xi^7\\
	&\quad + 110539728 \xi^8 - 62178597 \xi^9 + 14348907 \xi^{10},
\end{align*}
where we have applied the relations in Lemmas \ref{le:xi-IX-poly-ini} and \ref{le:gamma-xi-IX-poly-ini}. Therefore, by \eqref{eq:U-xi-poly} and \eqref{eq:U-gamma-xi-poly}, it is true that each $L_\alpha$ can be expressed in the polynomial ring $\mathbb{Z}[\xi]$. To capture this feature, we write
\begin{align}\label{eq:L-def-1}
	L(\xi,\alpha) := L_\alpha
\end{align}
with
\begin{align*}
	L(\xi,\alpha) \in \mathbb{Z}[\xi].
\end{align*}

\begin{lemma}\label{le:Phi-diff-poly}
	For any $\alpha\ge 1$,
	\begin{align}
		L(\xi,\alpha+2) - L(\xi,\alpha) \in \xi(1-\xi)\cdot \mathbb{Z}[\xi].
	\end{align}
\end{lemma}

\begin{proof}
	We know that
	\begin{align}\label{eq:Phi-i=1}
		L(\xi,3) - L(\xi,1) &= -(1-\xi)(87 \xi - 8802 \xi^2 + 231903 \xi^3 - 2598156 \xi^4\notag\\
		&\quad + 14309541 \xi^5 - 41157153 \xi^6 + 62710038 \xi^7\notag\\
		&\quad - 47829690 \xi^8 + 14348907 \xi^9),
	\end{align}
	which establishes the claimed result for $\alpha=1$. Now we assume the claim for a certain $2\alpha-1$. Namely,
	\begin{align*}
		L(\xi,2\alpha+1) - L(\xi,2\alpha-1) = (1-\xi)\sum_{j\ge 1} N_{2\alpha-1}(j)\xi^j.
	\end{align*}
	Then by \eqref{eq:L-U3-even},
	\begin{align*}
		L(\xi,2\alpha+2) - L(\xi,2\alpha) &= U_3\big(L(\xi,2\alpha+1)\big) - U_3\big(L(\xi,2\alpha-1)\big)\\
		& = U_3\left((1-\xi)\sum_{j\ge 1} N_{2\alpha-1}(j)\xi^j\right)\\
		& = \sum_{j\ge 1} N_{2\alpha-1}(j) \cdot\cT_j, 
	\end{align*}
	which is in $\xi(1-\xi)\cdot \mathbb{Z}[\xi]$ according to Lemma~\ref{le:T-div}. Let us write
	\begin{align*}
		L(\xi,2\alpha+2) - L(\xi,2\alpha) = (1-\xi)\sum_{j\ge 1} N_{2\alpha}(j)\xi^j.
	\end{align*}
	Then by \eqref{eq:L-U3-odd},
	\begin{align*}
		L(\xi,2\alpha+3) - L(\xi,2\alpha+1) &= U_3\big(\gamma L(\xi,2\alpha+2)\big) - U_3\big(\gamma L(\xi,2\alpha)\big)\\
		& = U_3\left(\gamma(1-\xi)\sum_{j\ge 1} N_{2\alpha}(j)\xi^j\right)\\
		& = \sum_{j\ge 1} N_{2\alpha}(j) \cdot\cS_j,
	\end{align*}
	which is also in $\xi(1-\xi)\cdot \mathbb{Z}[\xi]$ but this time by Lemma~\ref{le:S-div}. The required result follows from such induction.
\end{proof}

Let us write
\begin{align*}
	L(\xi,\alpha+2) - L(\xi,\alpha) = (1-\xi)\sum_{j\ge 1} N_{\alpha}(j)\xi^j.
\end{align*}

We are ready to prove Theorem~\ref{th:ph-mod3}.

\begin{proof}[Proof of Theorem~\ref{th:ph-mod3}]
	It is sufficient to show that for every $\alpha\ge 1$ and $j\ge 1$,
	\begin{align}
		\nu_3\big(N_{2\alpha-1}(j)\big) &\ge 3\alpha + j -3,\label{eq:N-odd-ineq}\\
		\nu_3\big(N_{2\alpha}(j)\big) &\ge 3\alpha + j -1.\label{eq:N-even-ineq}
	\end{align}
	
	Clearly, \eqref{eq:N-odd-ineq} is true for $\alpha=1$ in light of \eqref{eq:Phi-i=1}. Assume that \eqref{eq:N-odd-ineq} holds for a certain $\alpha\ge 1$. We first prove \eqref{eq:N-even-ineq} for $\alpha$ under this inductive hypothesis, and then \eqref{eq:N-odd-ineq} for $\alpha+1$.
	
	Note that as we have shown in the proof of Lemma~\ref{le:Phi-diff-poly},
	\begin{align*}
		(1-\xi) \sum_{j\ge 1} N_{2\alpha}(j) \xi^j &= L(\xi,2\alpha+2) - L(\xi,2\alpha)\\
		&= \sum_{j\ge 1} N_{2\alpha-1}(j) \cdot\cT_j\\
		&= \sum_{j\ge 1} N_{2\alpha-1}(j) \cdot (1-\xi) \sum_{k\ge 1} T_j(k) \xi^k.
	\end{align*}
	Hence,
	\begin{align*}
		\sum_{j\ge 1} N_{2\alpha}(j) \xi^j = \sum_{j\ge 1} N_{2\alpha-1}(j) \sum_{k\ge 1} T_j(k) \xi^k,
	\end{align*}
	so that
	\begin{align*}
		N_{2\alpha}(j) = \sum_{l\ge 1} N_{2\alpha-1}(l) T_l(j).
	\end{align*}
	This further implies that
	\begin{align*}
		\nu_3\big(N_{2\alpha}(j)\big) &\ge \min_{l\ge 1} \big\{\nu_3\big(N_{2\alpha-1}(l)\big)+\nu_3\big(T_l(j)\big)\big\}\\
		&\ge \min_{l\ge 1} \big\{\big(3\alpha+l-3\big)+\big(j-l+2\big)\big\}\\
		&= 3\alpha+j-1,
	\end{align*}
	thereby confirming \eqref{eq:N-even-ineq} for $\alpha$.
	
	We utilize once more what has been obtained in the proof of Lemma~\ref{le:Phi-diff-poly},
	\begin{align*}
		(1-\xi) \sum_{j\ge 1} N_{2\alpha+1}(j) \xi^j &= L(\xi,2\alpha+3) - L(\xi,2\alpha+1)\\
		&= \sum_{j\ge 1} N_{2\alpha}(j) \cdot\cS_j\\
		&= \sum_{j\ge 1} N_{2\alpha}(j) \cdot (1-\xi) \sum_{k\ge 1} S_j(k) \xi^k.
	\end{align*}
	Thus,
	\begin{align*}
		N_{2\alpha+1}(j) = \sum_{l\ge 1} N_{2\alpha}(l) S_l(j),
	\end{align*}
	and therefore, by recalling \eqref{eq:nu-S},
	\begin{align*}
		\nu_3\big(N_{2\alpha+1}(j)\big) &\ge \min_{l\ge 1} \big\{\nu_3\big(N_{2\alpha}(l)\big)+\nu_3\big(S_l(j)\big)\big\}\\
		&\ge \min_{l\ge 1} \big\{\big(3\alpha+l-1\big)+\big(j-l+1\big)\big\}\\
		&= 3\alpha+j,
	\end{align*}
	which is exactly \eqref{eq:N-odd-ineq} with $\alpha$ replaced by $\alpha+1$.
\end{proof}

\section{Powers of $5$}\label{sec:mod5}

The analysis in this section can be executed in the same way as that in Section~\ref{sec:mod3} but with more computational delicacy, and we mark each notation with a ``hat'' to indicate the correspondence.

We start by introducing a Hauptmodul corresponding to $\mathrm{X}_0(10)$, here again living at the cusp represented by $[0]_{10}$:
\begin{align}
	\hxi=\hxi(q)=\hxi(\tau):= \left(\frac{\varphi(-q^5)}{\varphi(-q)}\right)^2.
\end{align}

The initial expressions of $U_5(\hxi^i)$ for $1\le i\le 5$, starting with
\begin{align*}
	U_5(\hxi) = 11 \hxi - 60 \hxi^2 + 175 \hxi^3 - 250 \hxi^4 + 125 \hxi^5,
\end{align*}
have been established in \cite[pp.~715--716, eqs.~(2.1)--(2.5)]{Che2021}. For the reader's convenience, we have recorded them in a supplementary \textit{Mathematica} notebook  \cite{CST2025}.

For generic $\upsilon=\upsilon(q)\in \mathbb{C}[[q]]$, we write
\begin{align*}
	\hcY_i := U_5(\upsilon \hxi^i).
\end{align*}
In light of Newton's identity, the following recurrence can be shown directly; it was essentially given in the proof of \cite[p.~721, Theorem~2.2]{Che2021}.

\begin{theorem}\label{th:hxi-poly-rec}
	For any $i\ge 5$,
	\begin{align}\label{eq:hxi-poly-rec}
		\hcY_i &= \big(55 \hxi - 300 \hxi^2 + 875 \hxi^3 - 1250 \hxi^4 + 625 \hxi^5\big)\hcY_{i-1}\notag\\
		&\quad - \big(60 \hxi - 175 \hxi^2 + 250 \hxi^3 - 125 \hxi^4\big)\hcY_{i-2}\notag\\
		&\quad + \big(35 \hxi - 50 \hxi^2 + 25 \hxi^3\big) \hcY_{i-3}\notag\\
		&\quad - \big(10 \hxi - 5 \hxi^2\big) \hcY_{i-4}\notag\\
		&\quad + \hxi \hcY_{i-5}.
	\end{align}
\end{theorem}

It is notable that, by choosing $\upsilon = \hxi$ and recalling the initial expressions of $U_5(\hxi^i)$ for $1\le i\le 5$, we have, for any $i\ge 1$,
\begin{align}\label{eq:U-hxi-poly}
	U_5(\hxi^i) \in \hxi\cdot \mathbb{Z}[\hxi].
\end{align}

Now assume that we are given a generic family of polynomials $(\check{\cY}_i)_{i\ge i_0}$ in $\mathbb{Z}[\hxi]$ starting with a certain index $i_0\ge 0$. Furthermore, we require that these polynomials satisfy the recurrence in \eqref{eq:hxi-poly-rec}. Write
\begin{align*}
	\check{\cY}_i = \sum_{j\ge 0} \check{Y}_i(j) \hxi^j.
\end{align*}
Then the $5$-adic valuation of $\check{Y}_i(j)$ has the following property.

\begin{lemma}\label{le:check-Y}
	Let $\varepsilon\in \mathbb{Z}$ be given. If
	\begin{align}\label{eq:check-Y}
		\nu_5\big(\check{Y}_i(j)\big) \ge \left\lfloor \frac{5j-i+\varepsilon}{6} \right\rfloor
	\end{align}
	holds for every $i_0\le i\le i_0+4$ and $j\ge 0$, then the bound \eqref{eq:check-Y} also holds for all $i\ge i_0$.
\end{lemma}

\begin{proof}
	For convenience, we write the recurrence \eqref{eq:hxi-poly-rec} as
	\begin{align*}
		\check{\cY}_i = \sum_{n=1}^5 \sum_{m=1}^{5-n+1} c_{n,m} \hxi^m \check{\cY}_{i-n}.
	\end{align*}
	Now, with the initial bounds for $i_0\le i\le i_0+4$, we inductively have
	\begin{align*}
		\nu_5\big(\check{Y}_i(j)\big) &\ge \min_{n,m} \big\{\nu_5\big(c_{n,m}\big) + \nu_5\big(\check{Y}_{i-n}(j-m)\big)\big\}\\
		&\ge \min_{n,m} \left\{\nu_5\big(c_{n,m}\big) + \left\lfloor \frac{5(j-m)-(i-n)+\varepsilon}{6} \right\rfloor\right\}\\
		&\ge \left\lfloor \frac{5j-i+\varepsilon}{6} \right\rfloor.
	\end{align*}
	The bound then holds for all $i\ge i_0$.
\end{proof}

For $i\ge 1$, let
\begin{align*}
	\hcT_i:=U_5\big((1-\hxi)\hxi^i\big).
\end{align*}
By the initial expressions of $U_5(\hxi^i)$ in $\mathbb{Z}[\hxi]$ and the recurrence in Theorem~\ref{th:hxi-poly-rec}, we clearly have the following factorization of $\hcT_i$ in the polynomial ring $\mathbb{Z}[\hxi]$.

\begin{lemma}\label{le:hT-div}
	For every $i\ge 1$, we always have
	\begin{align*}
		\frac{\hcT_i}{\hxi(1-\hxi)}\in \mathbb{Z}[\hxi].
	\end{align*}
\end{lemma}

Let us write for $i\ge 1$,
\begin{align*}
	\hcT_i = (1-\hxi) \sum_{j\ge 1} \hT_i(j) \hxi^j.
\end{align*}
It is clear that the polynomials $\sum_{j\ge 1} \hT_i(j) \hxi^j$ also satisfy the recurrence in \eqref{eq:hxi-poly-rec}. After checking the initial $5$-adic valuations of $\hT_i(j)$ for $1\le i\le 5$, we arrive at the following bound according to Lemma~\ref{le:check-Y}.

\begin{theorem}\label{th:nu-hT}
	For any $i,j\ge 1$,
	\begin{align}\label{eq:nu-hT}
		\nu_5\big(\hT_i(j)\big) \ge \left\lfloor \frac{5j-i+3}{6} \right\rfloor.
	\end{align}
\end{theorem}

\subsection{Series $\Phi_{25}$}

Recall that
\begin{align*}
	\sum_{n\ge 0}\ph_{25}(n)q^n =\dfrac{\varphi(-q^{25})}{\varphi(-q)}.
\end{align*}

\begin{lemma}\label{le:ph25-xi}
	For any $\alpha\ge 1$,
	\begin{align}\label{eq:ph25-xi}
		\sum_{n\ge 0}\ph_{25}\big(5^\alpha n\big)q^n = U_5^{(\alpha-1)}(\hxi - 5 \hxi^2 + 5 \hxi^3).
	\end{align}
\end{lemma}

\begin{proof}
	It suffices to show
	\begin{align*}
		\sum_{n\ge 0}\ph_{25}(5 n)q^n = U_5\left(\dfrac{\varphi(-q^{25})}{\varphi(-q)}\right) = \hxi - 5 \hxi^2 + 5 \hxi^3,
	\end{align*}
	which can be established by a standard cusp analysis.
\end{proof}

Recalling \eqref{eq:U-hxi-poly}, we note that $\sum_{n\ge 0}\ph_{25}\big(5^\alpha n\big)q^n$ can be expressed as a polynomial in $\mathbb{Z}[\hxi]$ for every $\alpha\ge 1$. Therefore, we define
\begin{align}\label{eq:hL25-def-1}
	\hL_{25}(\hxi,\alpha) := \sum_{n\ge 0} \ph_{25} \big(5^{\alpha} n\big) q^n
\end{align}
with
\begin{align*}
	\hL_{25}(\hxi,\alpha) \in \mathbb{Z}[\hxi].
\end{align*}

We further compute that
\begin{align*}
	\hL_{25}(\hxi,2) - \hL_{25}(\hxi,1) = U_5(\hxi - 5 \hxi^2 + 5 \hxi^3) - (\hxi - 5 \hxi^2 + 5 \hxi^3)
\end{align*}
indeed lies in $\hxi(1-\hxi)\cdot \mathbb{Z}[\hxi]$. The next relation then follows from Lemma~\ref{le:hT-div}.

\begin{lemma}
	For any $\alpha\ge 1$,
	\begin{align}
		\hL_{25}(\hxi,\alpha+1) - \hL_{25}(\hxi,\alpha) \in \hxi(1-\hxi)\cdot \mathbb{Z}[\hxi].
	\end{align}
\end{lemma}

Let us write
\begin{align*}
	\hL_{25}(\hxi,\alpha+1) - \hL_{25}(\hxi,\alpha) = (1-\hxi) \sum_{j\ge 1} \hZ_i(j) \hxi^j.
\end{align*}

We are ready to prove Theorem~\ref{th:ph25-mod5}.

\begin{proof}[Proof of Theorem~\ref{th:ph25-mod5}]
	It is sufficient to show that for every $\alpha\ge 1$ and $j\ge 1$,
	\begin{align}\label{eq:hZ-ineq}
		\nu_5\big(\hZ_\alpha(j)\big) \ge \alpha + \left\lfloor \frac{5j-5}{6} \right\rfloor.
	\end{align}
	The case of $\alpha=1$ can be shown by a direct verification. For $\alpha\ge 2$, we have
	\begin{align*}
		\hZ_\alpha(j) = \sum_{l\ge 1} \hZ_{\alpha-1}(l) \hT_l(j)
	\end{align*}
	so that
	\begin{align*}
		\nu_5\big(\hZ_\alpha(j)\big) &\ge \min_{l\ge 1} \big\{\nu_5\big(\hZ_{\alpha-1}(l)\big)+\nu_5\big(\hT_l(j)\big)\big\}\\
		&\ge \min_{l\ge 1} \left\{\left((\alpha-1)+\left\lfloor \frac{5l-5}{6} \right\rfloor\right)+\left\lfloor \frac{5j-l+3}{6} \right\rfloor\right\}\\
		&\ge \alpha + \left\lfloor \frac{5j-5}{6} \right\rfloor.
	\end{align*}
	For the last inequality, we may first verify it directly for the cases $l=1,2,3$. Furthermore, when $l\ge 4$,
	\begin{align*}
		&\left((\alpha-1)+\left\lfloor \frac{5l-5}{6} \right\rfloor\right)+\left\lfloor \frac{5j-l+3}{6} \right\rfloor\\
		&\qquad \ge \left((\alpha-1) + \frac{5l-5}{6} - \frac{5}{6}\right) + \left(\frac{5j-l+3}{6} - \frac{5}{6}\right)\\
		&\qquad \ge \alpha + \left\lfloor \frac{5j-5}{6} \right\rfloor.
	\end{align*}
	The desired inequality \eqref{eq:hZ-ineq} then follows by induction on $\alpha$.
\end{proof}

\subsection{Series $\Phi$}

For this moment, we define
\begin{align*}
	\hgamma:=\Phi_{25}=\frac{\varphi(-q^{25})}{\varphi(-q)}.
\end{align*}
We may also represent the initial cases of $U_5(\hgamma \hxi^i)$ in $\mathbb{Z}[\hxi]$ for $0\le i\le 4$, as recorded in the supplementary \textit{Mathematica} notebook \cite{CST2025}. The first among them is
\begin{align*}
	U_5(\hgamma) = \hxi - 5 \hxi^2 + 5 \hxi^3,
\end{align*}
which already appears in the proof of Lemma~\ref{le:ph25-xi}.

In light of these retaions, we choose $\upsilon = \hgamma$ in Theorem~\ref{th:hxi-poly-rec} and find that for any $i\ge 0$,
\begin{align}\label{eq:U-hgamma-hxi-poly}
	U_5(\hgamma\hxi^i) \in \hxi\cdot \mathbb{Z}[\hxi].
\end{align}

Next, for $i\ge 0$, let
\begin{align*}
	\hcS_i:=U_5\big(\hgamma(1-\hxi)\hxi^i\big).
\end{align*}
The following relation is also clear.

\begin{lemma}\label{le:hS-div}
	For every $i\ge 0$, we always have
	\begin{align*}
		\frac{\hcS_i}{\hxi(1-\hxi)}\in \mathbb{Z}[\hxi].
	\end{align*}
\end{lemma}

We write for $i\ge 0$,
\begin{align*}
	\hcS_i = (1-\hxi) \sum_{j\ge 1} \hS_i(j) \hxi^j.
\end{align*}
It is again clear that the polynomials $\sum_{j\ge 1} \hS_i(j) \hxi^j$ satisfy the recurrence in \eqref{eq:hxi-poly-rec}. After checking the initial $5$-adic valuations of $\hS_i(j)$ for $1\le i\le 5$, we arrive at the following bound according to Lemma~\ref{le:check-Y}.

\begin{theorem}\label{th:nu-hS}
	For any $i,j\ge 1$,
	\begin{align}\label{eq:nu-hS}
		\nu_5\big(\hS_i(j)\big) \ge \left\lfloor \frac{5j-i}{6} \right\rfloor.
	\end{align}
\end{theorem}

Recall that
\begin{align*}
	\sum_{n\ge 0}\ph(n)q^n =\dfrac{1}{\varphi(-q)}.
\end{align*}

For each $\alpha\ge 1$, we define
\begin{align*}
	\hL_\alpha := \begin{cases}
		\displaystyle\varphi(-q^5)\sum_{n\ge 0} \ph \big(5^{\alpha} n\big) q^n, & \text{if $\alpha$ is odd},\\[20pt]
		\displaystyle\varphi(-q)\sum_{n\ge 0} \ph \big(5^{\alpha} n\big) q^n, & \text{if $\alpha$ is even}.
	\end{cases}
\end{align*}
It is clear that for every $\alpha\ge 1$,
\begin{align}
	\hL_{2\alpha} &= U_5(\hL_{2\alpha-1}),\label{eq:hL-U5-even}\\
	\hL_{2\alpha+1} &= U_5(\hgamma\hL_{2\alpha}).\label{eq:hL-U5-odd}
\end{align}
Taking into account the fact that $\hL_1 = U_5(\hgamma)\in \mathbb{Z}[\hxi]$, we conclude from \eqref{eq:U-hxi-poly} and \eqref{eq:U-hgamma-hxi-poly} that each $\hL_\alpha$ can be expressed in the polynomial ring $\mathbb{Z}[\hxi]$. We emphasize this feature by writing
\begin{align}\label{eq:hL-def-1}
	\hL(\hxi,\alpha) := \hL_\alpha
\end{align}
with
\begin{align*}
	\hL(\hxi,\alpha) \in \mathbb{Z}[\hxi].
\end{align*}

Next, we compute that
\begin{align}\label{eq:hPhi-3-1}
	\hL(\hxi,3) - \hL(\hxi,1) &= U_5\big(\hgamma U_5\big(U_5(\hgamma)\big)\big) - U_5(\hgamma)\notag\\
	&= U_5\big(\hgamma U_5\big(U_5(\hxi - 5 \hxi^2 + 5 \hxi^3)\big)\big) - (\hxi - 5 \hxi^2 + 5 \hxi^3)
\end{align}
is indeed in $\hxi(1-\hxi)\cdot \mathbb{Z}[\hxi]$. We then apply the same inductive argument as in the proof of Lemma~\ref{le:Phi-diff-poly} to get the following relation.

\begin{lemma}\label{le:hPhi-diff-poly}
	For any $\alpha\ge 1$,
	\begin{align}
		\hL(\hxi,\alpha+2) - \hL(\hxi,\alpha) \in \hxi(1-\hxi)\cdot \mathbb{Z}[\hxi].
	\end{align}
\end{lemma}

Let us write
\begin{align*}
	\hL(\hxi,\alpha+2) - \hL(\hxi,\alpha) = (1-\hxi)\sum_{j\ge 1} \hN_{\alpha}(j)\hxi^j.
\end{align*}

We are ready to prove Theorem~\ref{th:ph-mod5}.

\begin{proof}[Proof of Theorem~\ref{th:ph-mod5}]
	It is sufficient to show that for every $\alpha\ge 1$ and $j\ge 1$,
	\begin{align}
		\nu_5\big(\hN_{2\alpha-1}(j)\big) &\ge \alpha + \left\lfloor \frac{5j-6}{6} \right\rfloor,\label{eq:hN-odd-ineq}\\
		\nu_5\big(\hN_{2\alpha}(j)\big) &\ge \alpha + \left\lfloor \frac{5j-4}{6} \right\rfloor.\label{eq:hN-even-ineq}
	\end{align}
	Note that \eqref{eq:hN-odd-ineq} suggests why unlike Theorem~\ref{th:ph-mod3}, we only have the congruence family for even powers in Theorem~\ref{th:ph-mod5} --- Taking $j=1$ in \eqref{eq:hN-odd-ineq} gives $\alpha+\lfloor \frac{5\times 1-6}{6} \rfloor$, which equals $\alpha-1$, smaller than $\alpha$.
	
	Now we prove \eqref{eq:hN-odd-ineq} and \eqref{eq:hN-even-ineq} by induction on $\alpha$. It is easy to check \eqref{eq:hN-odd-ineq} for $\alpha=1$ when computing \eqref{eq:hPhi-3-1}. Assume that \eqref{eq:hN-odd-ineq} holds for a certain $\alpha\ge 1$. We first prove \eqref{eq:hN-even-ineq} for $\alpha$ under this inductive hypothesis, and then \eqref{eq:hN-odd-ineq} for $\alpha+1$, in the manner of the proof of \eqref{eq:N-odd-ineq} and \eqref{eq:N-even-ineq}. Note that
	\begin{align*}
		\hN_{2\alpha}(j) = \sum_{l\ge 1} \hN_{2\alpha-1}(l) \hT_l(j).
	\end{align*}
	This, with \eqref{eq:nu-hT} recalled, implies that
	\begin{align*}
		\nu_5\big(\hN_{2\alpha}(j)\big) &\ge \min_{l\ge 1} \big\{\nu_5\big(\hN_{2\alpha-1}(l)\big)+\nu_5\big(\hT_l(j)\big)\big\}\\
		&\ge \min_{l\ge 1} \left\{\left(\alpha + \left\lfloor \frac{5l-6}{6} \right\rfloor\right)+\left\lfloor \frac{5j-l+3}{6} \right\rfloor\right\}\\
		&\ge \alpha + \left\lfloor \frac{5j-4}{6} \right\rfloor,
	\end{align*}
	and hence establishes \eqref{eq:hN-even-ineq}. Furthermore,
	\begin{align*}
		\hN_{2\alpha+1}(j) = \sum_{l\ge 1} \hN_{2\alpha}(l) \hS_l(j),
	\end{align*}
	and therefore, by using \eqref{eq:nu-hS},
	\begin{align*}
		\nu_5\big(\hN_{2\alpha+1}(j)\big) &\ge \min_{l\ge 1} \big\{\nu_5\big(\hN_{2\alpha}(l)\big)+\nu_5\big(\hS_l(j)\big)\big\}\\
		&\ge \min_{l\ge 1} \left\{\left(\alpha + \left\lfloor \frac{5l-4}{6} \right\rfloor\right)+\left\lfloor \frac{5j-l}{6} \right\rfloor\right\}\\
		&\ge \alpha + \left\lfloor \frac{5j}{6} \right\rfloor,
	\end{align*}
	which is exactly \eqref{eq:hN-odd-ineq} with $\alpha$ replaced by $\alpha+1$.
\end{proof}

\section{Atkin--Lehner involution}

It remains to show the four families of internal congruences related to $\psi(q)$. This is accomplished by demonstrating that these congruences are equivalent to those of $\varphi(-q)$; that is, there exists a multiplicity in these congruences which includes those of $\psi(q)$.

To do this, as in \cite{GSS2024} and \cite{SellersS0}, we can study the action of a certain Atkin--Lehner involution, originally introduced in \cite{GSS2024}.  The theory of such mappings was developed in \cite{AtkinL}, originally for the purpose of studying the properties of cusp forms.  We will give a definition matching those of \cite[p.~27, Definition~2.19]{Ono0} and \cite[p.~8, Definition~3.2]{SellersS0}.

Recall that the \emph{Dedekind eta function} is defined by
\begin{align*}
	\eta(\tau) := q^{1/24}\prod_{k\ge 1} (1-q^k),
\end{align*}
where $q:=e^{2\pi i \tau}$ with $\tau\in \mathbb{H}$. This is a modular form, whose critical property can be found in \cite[p.~52, Theorem~3.4]{Apo1990}.

Let
\begin{align*}
	\kappa:=\begin{pmatrix}
		a & b\\
		c & d
	\end{pmatrix} \in \operatorname{SL}_2(\mathbb{Z})
\end{align*}
be such that $c>0$. Then recalling the action (\ref{sl2action}), we have
\begin{align*}
	\eta(\kappa\tau) = \epsilon(\kappa) \big({-i}(c\tau+d)\big)^{1/2} \eta(\tau),
\end{align*}
wherein we take the principal branch of $z^{1/2}$, and $\epsilon$ is a certain $24$-th root of unity:
\begin{align*}
	\epsilon(\kappa) := e^{\pi i \left(\frac{a+d}{12c}-s(d,c)\right)}
\end{align*}
with $s(d,c)$ being the \emph{Dedekind sum}:
\begin{align*}
	s(h,k) := \sum_{m=1}^{k-1} \frac{m}{k} \left(\frac{hm}{k}-\left\lfloor\frac{hm}{k}\right\rfloor-\frac{1}{2}\right).
\end{align*}

The associated action of $\kappa\in\operatorname{SL}_2(\mathbb{Z})$ can naturally be extended to include elements of $\operatorname{GL}_2(\mathbb{Z})$.  With that, we define the following concept:

\begin{definition}\label{ALdefn}
Let $N$ be a positive integer, $p$ a prime divisor of $N$, and $q = p^k$ the maximum power of $p$ which divides $N$.  The associated \emph{Atkin--Lehner involution} for $\Gamma_0(N)$ is the action induced by the matrix
\begin{align*}
W_{N,q}:=\begin{pmatrix}
qA & B\\
NC & qD
\end{pmatrix},
\end{align*} in which $A,B,C,D\in\mathbb{Z}$, and $\mathrm{det}\left(W_{N,q}\right)=q$.
\end{definition}

It is important to note that for a choice of fixed $N$, $p$ and $k$, these operators are well-defined: we may choose the integers $A,B,C,D$ however we like, so long as the determinant remains $q$; see \cite[p.~139, Lemma~10]{AtkinL} and \cite[p.~27, Remark~2.20]{Ono0}.

\subsection{Powers of $3$}\label{powersof3etc}

In Section~\ref{sec:mod3}, we have introduced a Hauptmodul $\xi$ corresponding to $\mathrm{X}_0(6)$ at $[0]_6$ to work out the internal congruence families related to $\varphi(-q)$ modulo powers of $3$. Writing $\xi$ in terms of the Dedekind eta function, we have
\begin{align}\label{eq:Hauptmodul-xi}
	\xi = \xi(\tau) := \left(\frac{\varphi(-q^3)}{\varphi(-q)}\right)^4 =  \frac{\eta(2\tau)^4 \eta(3\tau)^8}{\eta(\tau)^8 \eta(6\tau)^4}.
\end{align}
Now to deal with the congruences related to $\psi(q)$, we need a second Hauptmodul on $\mathrm{X}_0(6)$, living at $[1/2]_{6}$:
\begin{align}\label{eq:Hauptmodul-zeta}
	\zeta = \zeta(\tau) := q\left(\frac{\psi(q^3)}{\psi(q)}\right)^4 =  \frac{\eta(\tau)^4 \eta(6\tau)^8}{\eta(2\tau)^8 \eta(3\tau)^4}.
\end{align}
Recall also that
\begin{align*}
	\Phi_9(\tau) &= \frac{\varphi(-q^9)}{\varphi(-q)} = \frac{\eta(2\tau)\eta(9\tau)^2}{\eta(\tau)^2 \eta(18\tau)},\\
	q\Psi_9(\tau) &= q\frac{\psi(q^9)}{\psi(q)} = \frac{\eta(\tau)\eta(18\tau)^2}{\eta(2\tau)^2 \eta(9\tau)}.
\end{align*}

We construct the matrix
\begin{align}\label{eq:V-def}
	V := \begin{pmatrix} 2 & -1 \\ 54 & -26 \end{pmatrix},
\end{align} which satisfies the conditions of Definition~\ref{ALdefn} above, for $N=6$ and $p=2$.
We want to apply the action of $V$ to $\xi$.  We can of course do this by applying $V$ to each eta factor of $\xi$.  Letting $\delta \in \mathbb{Z}$, we write for convenience
\begin{align*}
	\delta V = \begin{pmatrix} \delta & 0 \\ 0 & 1 \end{pmatrix} \begin{pmatrix} 2 & -1 \\ 54 & -26 \end{pmatrix} = \begin{pmatrix} 2\delta & -\delta \\ 54 & -26 \end{pmatrix}.
\end{align*}
Finally, note that we can decompose this matrix in the following way:
\begin{align*}
	\delta V = \begin{pmatrix} 2\delta & -\delta \\ 54 & -26 \end{pmatrix} = \begin{pmatrix} 2\delta/g & -y \\ 54/g & x \end{pmatrix}\begin{pmatrix} g & \beta \\ 0 & 2\delta/g \end{pmatrix},
\end{align*} with $g = \mathrm{gcd}(2\delta,54)=2\mathrm{gcd}(\delta,27)$, $\beta\in\mathbb{Z}$, and $x,y\in\mathbb{Z}$ such that our first factor has determinant $1$.

With this simple matrix algebra, we can decompose our involution into modular transformations for each eta factor.  We prove the following two crucial relations.

\begin{lemma}\label{le:V-xi}
	We have
	\begin{align}\label{eq:V-xi}
		\xi(V\tau) = \zeta(\tau).
	\end{align}
\end{lemma}

\begin{proof}
	Note that
	\begin{align*}
		\xi(V\tau) &= \frac{\eta(2V\tau)^4 \eta(3V\tau)^8}{\eta(V\tau)^8 \eta(6V\tau)^4}.
	\end{align*}
	In terms of the action of $\operatorname{SL}_2(\mathbb{Z})$, we have
	\begin{align*}
		V\tau &= \begin{pmatrix} 1 & 0\\27 & 1 \end{pmatrix} (2\tau-1) =: V_1(2\tau-1),\\
		2V\tau &= \begin{pmatrix} 2 & -1\\27 & -13 \end{pmatrix} (\tau) =: V_2(\tau),\\
		3V\tau &= \begin{pmatrix} 1 & 0\\9 & 1 \end{pmatrix} (6\tau-3) =: V_3(6\tau-3),\\
		6V\tau &= \begin{pmatrix} 2 & 1\\9 & 5 \end{pmatrix} (3\tau-2) =: V_6(3\tau-2).
	\end{align*}
	Therefore, by invoking the modular transformation for the Dedekind eta function, we have
	\begin{align*}
		\xi(V\tau) &= \frac{\epsilon(V_2)^4 \epsilon(V_3)^8}{\epsilon(V_1)^8 \epsilon(V_6)^4}\cdot  \frac{\eta(\tau)^4 \eta(6\tau-3)^8}{\eta(2\tau-1)^8 \eta(3\tau-2)^4}\\
		&= \frac{\eta(\tau)^4 \eta(6\tau)^8}{\eta(2\tau)^8 \eta(3\tau)^4}\\
		&= \zeta(\tau),
	\end{align*}
	as claimed.
\end{proof}

\begin{lemma}\label{le:V-Ph}
	We have
	\begin{align}\label{eq:V-Ph}
		\Phi_9(V\tau) = q\Psi_9(\tau).
	\end{align}
\end{lemma}

\begin{proof}
	Note that
	\begin{align*}
		\Phi_9(V\tau) = \frac{\eta(2V\tau)\eta(9V\tau)^2}{\eta(V\tau)^2 \eta(18V\tau)}.
	\end{align*}
	Apart from the rewriting of $V\tau$ and $2V\tau$ in terms of the action of $\operatorname{SL}_2(\mathbb{Z})$ as before, we also require
	\begin{align*}
		9V\tau &= \begin{pmatrix} 1 & 0\\3 & 1 \end{pmatrix} (18\tau-9) =: V_9(18\tau-9),\\
		18V\tau &= \begin{pmatrix} 2 & 3\\3 & 5 \end{pmatrix} (9\tau-6) =: V_{18}(9\tau-6).
	\end{align*}
	Thus,
	\begin{align*}
		\Phi_9(V\tau) &= \frac{\epsilon(V_2) \epsilon(V_9)^2}{\epsilon(V_1)^2 \epsilon(V_{18})}\cdot \frac{\eta(\tau)\eta(18\tau-9)^2}{\eta(2\tau-1)^2 \eta(9\tau-6)}\\
		&= \frac{\eta(\tau)\eta(18\tau)^2}{\eta(2\tau)^2 \eta(9\tau)}\\
		&= q\Psi_9(\tau),
	\end{align*}
	as claimed.
\end{proof}

\subsubsection{Series $\Phi_9$ and $\Psi_9$}\label{sec:ps9-mod3}

To prove the internal congruences for $\Phi_9$ in Theorem~\ref{th:ph9-mod3}, the most important ingredient is the family of polynomial representations in \eqref{eq:L9-def-1}:
\begin{align*}
	L_9(\xi,\alpha) = \sum_{n\ge 0} \ph_9 \big(3^{\alpha} n\big) q^n,
\end{align*}
whereby each $L_9(\xi,\alpha)\in \mathbb{Z}[\xi]$. In other words,
\begin{align}\label{eq:L9-U3-relation}
	L_9(\xi,\alpha) = U_3^{(\alpha)} (\Phi_9).
\end{align}
Now we lift them to a family of polynomials $L_9(x,\alpha)\in \mathbb{Z}[x]$ by replacing $\xi$ with a generic indeterminate $x$. In light of \eqref{eq:Z-ineq},
\begin{align}
	L_9(x,\alpha+1)- L_9(x,\alpha) \equiv 0 \pmod{3^{2\alpha}}.
\end{align}
Equipped with this family of congruences, we shall see that Theorem~\ref{th:ps9-mod3} is an immediate consequence of the following relation.

\begin{theorem}\label{th:L9-zeta}
	For any $\alpha\ge 1$,
	\begin{align}
		L_9(\zeta,\alpha) = \sum_{n\ge 0} \ps_9 \big(3^{\alpha} n + 3^{\alpha} -1\big) q^{n+1}.
	\end{align}
\end{theorem}

\begin{proof}
	Noting the fact that the operations of $V$ and $U_3$ are interchangeable, we have
	\begin{align*}
		\sum_{n\ge 0} \ps_9 \big(3^{\alpha} n + 3^{\alpha} -1\big) q^{n+1} &=U_{3}^{(\alpha)}( q\Psi_9 ) \overset{\eqref{eq:V-Ph}}{=} U_{3}^{(\alpha)}\big( \Phi_9(V\tau) \big)\\
		&= U_{3}^{(\alpha)}( \Phi_9 )(V\tau)\overset{\eqref{eq:L9-U3-relation}}{=} L_9(\xi,\alpha)(V\tau)\\
		&= L_9(\xi(V\tau),\alpha)\overset{\eqref{eq:V-xi}}{=} L_9(\zeta,\alpha),
	\end{align*}
	as desired.
\end{proof}

\subsubsection{Series $\Phi$ and $\Psi$}

The internal congruences for $\Phi$ in Theorem~\ref{th:ph-mod3} rely on the representations in \eqref{eq:L-def-1}:
\begin{align*}
	L(\xi,2\alpha-1) &= \varphi(-q^3)\sum_{n\ge 0} \ph \big(3^{2\alpha-1} n\big)q^n,\\
	L(\xi,2\alpha) &= \varphi(-q)\sum_{n\ge 0} \ph \big(3^{2\alpha} n\big) q^n,
\end{align*}
whereby each $L(\xi,\alpha) \in \mathbb{Z}[\xi]$. In particular, \eqref{eq:L-U3-even} and \eqref{eq:L-U3-odd} tell us that
\begin{align}
	L(\xi,2\alpha) &= U_3\big(L(\xi,2\alpha-1)\big),\label{eq:L-U3-even-1}\\
	L(\xi,2\alpha+1) &= U_3\big(\Phi_9\cdot L(\xi,2\alpha)\big).\label{eq:L-U3-odd-1}
\end{align}
We similarly lift them to a family of polynomials $L(x,\alpha)\in \mathbb{Z}[x]$ by replacing $\xi$ with an indeterminate $x$. By \eqref{eq:hN-odd-ineq} and \eqref{eq:hN-even-ineq},
\begin{align}
	L(x,2\alpha+1)- L(x,2\alpha-1) &\equiv 0 \pmod{3^{3\alpha-2}},\\
	L(x,2\alpha+2)- L(x,2\alpha) &\equiv 0 \pmod{3^{3\alpha}}.
\end{align}
Therefore, the congruences in Theorem~\ref{th:ps-mod3} follow directly from the next result.

\begin{theorem}
	For any $\alpha\ge 1$,
	\begin{align}
		L(\zeta,2\alpha-1) &= \psi(q^3)\sum_{n\ge 0} \ps \!\left(3^{2\alpha-1} n + \frac{5\cdot 3^{2\alpha-1}+1}{8}\right)q^{n+1},\label{eq:L-zeta-ps-odd}\\
		L(\zeta,2\alpha) &= \psi(q)\sum_{n\ge 0} \ps \!\left(3^{2\alpha} n + \frac{7\cdot 3^{2\alpha}+1}{8}\right) q^{n+1}.\label{eq:L-zeta-ps-even}
	\end{align}
\end{theorem}

\begin{proof}
	We begin with the $\alpha=1$ case of \eqref{eq:L-zeta-ps-odd}:
	\begin{align*}
		\psi(q^3) \sum_{n\ge 0} \ps(3n+2)q^{n+1} = U_3(q\Psi_9) = L_9(\zeta,1) = L(\zeta,1),
	\end{align*}
	where the second equality has been shown in the proof of Theorem~\ref{th:L9-zeta} and the last equality comes from the fact that
	\begin{align*}
		L(x,1) = L_9(x,1),
	\end{align*}
	which further follows from the observation that
	\begin{align*}
		L(\xi,1) = \varphi(-q^3)\sum_{n\ge 0} \ph (3n)q^n = U_3(\Phi_9) = \sum_{n\ge 0} \ph_9 (3n) q^n = L_9(\xi,1).
	\end{align*}
	
	Now we assume \eqref{eq:L-zeta-ps-odd} for a certain $\alpha\ge 1$. Then
	\begin{align*}
		L(\zeta,2\alpha) &= L(\xi,2\alpha)(V\tau) \overset{\eqref{eq:L-U3-even-1}}{=} U_3\big(L(\xi,2\alpha-1)\big)(V\tau)\\
		& = U_3\big(L(\xi(V\tau),2\alpha-1)\big) \overset{\eqref{eq:V-xi}}{=} U_3\big(L(\zeta,2\alpha-1)\big).
	\end{align*}
	Hence,
	\begin{align*}
		L(\zeta,2\alpha) &= U_3\left(\psi(q^3)\sum_{n\ge 0} \ps \!\left(3^{2\alpha-1} n + \frac{5\cdot 3^{2\alpha-1}+1}{8}\right)q^{n+1}\right)\\
		&= \psi(q)\sum_{n\ge 0} \ps \!\left(3^{2\alpha} n + \frac{7\cdot 3^{2\alpha}+1}{8}\right) q^{n+1},
	\end{align*}
	which gives \eqref{eq:L-zeta-ps-even} for $\alpha$.
	
	Finally,
	\begin{align*}
		L(\zeta,2\alpha+1) &= L(\xi,2\alpha+1)(V\tau) \overset{\eqref{eq:L-U3-odd-1}}{=} U_3\big(\Phi_9\cdot L(\xi,2\alpha)\big)(V\tau)\\
		&= U_3\big(\Phi_9(V\tau)\cdot L(\xi(V\tau),2\alpha)\big) \overunderset{\eqref{eq:V-xi}}{\eqref{eq:V-Ph}}{=} U_3\big(q\Psi_9\cdot L(\zeta,2\alpha)\big).
	\end{align*}
	We conclude that
	\begin{align*}
		L(\zeta,2\alpha+1) &= U_3\left(q\frac{\psi(q^9)}{\psi(q)}\cdot \psi(q)\sum_{n\ge 0} \ps \!\left(3^{2\alpha} n + \frac{7\cdot 3^{2\alpha}+1}{8}\right) q^{n+1}\right)\\
		&= \psi(q^3)\sum_{n\ge 0} \ps \!\left(3^{2\alpha+1} n + \frac{5\cdot 3^{2\alpha+1}+1}{8}\right)q^{n+1},
	\end{align*}
	which is exactly \eqref{eq:L-zeta-ps-odd} with $\alpha$ replaced by $\alpha+1$.
\end{proof}

\subsection{Powers of $5$}

The Hauptmodul $\hxi$ on $\mathrm{X}_0(10)$ at $[0]_{10}$ introduced in Section~\ref{sec:mod5} is
\begin{align}\label{eq:Hauptmodul-hxi}
	\hxi = \hxi(\tau) := \left(\frac{\varphi(-q^5)}{\varphi(-q)}\right)^2 =  \frac{\eta(2\tau)^2 \eta(5\tau)^4}{\eta(\tau)^4 \eta(10\tau)^2}.
\end{align}
Now we construct a second Hauptmodul on $\mathrm{X}_0(10)$, which lives at the cusp $[1/2]_{10}$:
\begin{align}\label{eq:Hauptmodul-hzeta}
	\hzeta = \hzeta(\tau) := q\left(\frac{\psi(q^5)}{\psi(q)}\right)^2 =  \frac{\eta(\tau)^2 \eta(10\tau)^4}{\eta(2\tau)^4 \eta(5\tau)^2}.
\end{align}
Recall also that
\begin{align*}
	\Phi_{25}(\tau) &= \frac{\varphi(-q^{25})}{\varphi(-q)} = \frac{\eta(2\tau)\eta(25\tau)^2}{\eta(\tau)^2 \eta(50\tau)},\\
	q^3\Psi_{25}(\tau) &= q^3\frac{\psi(q^{25})}{\psi(q)} = \frac{\eta(\tau)\eta(50\tau)^2}{\eta(2\tau)^2 \eta(25\tau)}.
\end{align*}

In the interests of relating our two Hauptmoduln, and again in similar form to the arguments in Section \ref{powersof3etc} above, we construct a matrix 
\begin{align}\label{eq:hV-def}
	\hV := \begin{pmatrix} 2 & 1 \\ 50 & 26 \end{pmatrix},
\end{align}
which satisfies Definition~\ref{ALdefn} for $N=10$ and $p=2$.  Again, for $\delta\in \mathbb{Z}$, we write
\begin{align*}
	\delta \hV = \begin{pmatrix} \delta & 0 \\ 0 & 1 \end{pmatrix} \begin{pmatrix} 2 & 1 \\ 50 & 26 \end{pmatrix} = \begin{pmatrix} 2\delta & \delta \\ 50 & 26 \end{pmatrix},
\end{align*}
which we can once again factor into an element of $\operatorname{SL}_2(\mathbb{Z})$ and an upper triangular element.

Our task for this moment is to establish analogs to \eqref{eq:V-xi} and \eqref{eq:V-Ph}.

\begin{lemma}\label{le:hV-hxi}
	We have
	\begin{align}\label{eq:hV-hxi}
		\hxi(\hV\tau) = \hzeta(\tau).
	\end{align}
\end{lemma}

\begin{proof}
	Note that
	\begin{align*}
		\hxi(\hV\tau) &= \frac{\eta(2\hV\tau)^2 \eta(5\hV\tau)^4}{\eta(\hV\tau)^4 \eta(10\hV\tau)^2}.
	\end{align*}
	In terms of the action of $\operatorname{SL}_2(\mathbb{Z})$, we have
	\begin{align*}
		\hV\tau &= \begin{pmatrix} 1 & 0\\25 & 1 \end{pmatrix} (2\tau+1) =: \hV_1(2\tau+1),\\
		2\hV\tau &= \begin{pmatrix} 2 & 1\\25 & 13 \end{pmatrix} (\tau) =: \hV_2(\tau),\\
		5\hV\tau &= \begin{pmatrix} 1 & 0\\5 & 1 \end{pmatrix} (10\tau+5) =: \hV_5(10\tau+5),\\
		10\hV\tau &= \begin{pmatrix} 2 & 1\\5 & 3 \end{pmatrix} (5\tau+2) =: \hV_{10}(5\tau+2).
	\end{align*}
	Therefore, by invoking the modular transformation for the Dedekind eta function, we have
	\begin{align*}
		\hxi(\hV\tau) &= \frac{\epsilon(\hV_2)^2 \epsilon(\hV_5)^4}{\epsilon(\hV_1)^4 \epsilon(\hV_{10})^2}\cdot  \frac{\eta(\tau)^2 \eta(10\tau+5)^4}{\eta(2\tau+1)^4 \eta(5\tau+2)^2}\\
		&= \frac{\eta(\tau)^2 \eta(10\tau)^4}{\eta(2\tau)^4 \eta(5\tau)^2}\\
		&= \hzeta(\tau),
	\end{align*}
	as claimed.
\end{proof}

\begin{lemma}\label{le:hV-hPh}
	We have
	\begin{align}\label{eq:hV-hPh}
		\Phi_{25}(\hV\tau) = q^3 \Psi_{25}(\tau).
	\end{align}
\end{lemma}

\begin{proof}
	Note that
	\begin{align*}
		\Phi_{25}(\hV\tau) = \frac{\eta(2\hV\tau)\eta(25\hV\tau)^2}{\eta(\hV\tau)^2 \eta(50\hV\tau)}.
	\end{align*}
	Apart from the rewriting of $\hV\tau$ and $2\hV\tau$ in terms of the action of $\operatorname{SL}_2(\mathbb{Z})$ as before, we also require
	\begin{align*}
		25\hV\tau &= \begin{pmatrix} 1 & 0\\1 & 1 \end{pmatrix} (50\tau+25) =: \hV_{25}(50\tau+25),\\
		50\hV\tau &= \begin{pmatrix} 2 & 5\\1 & 3 \end{pmatrix} (25\tau+10) =: \hV_{50}(25\tau+10).
	\end{align*}
	Thus,
	\begin{align*}
		\Phi_{25}(\hV\tau) &= \frac{\epsilon(\hV_2) \epsilon(\hV_{25})^2}{\epsilon(\hV_1)^2 \epsilon(\hV_{50})}\cdot \frac{\eta(\tau)\eta(50\tau+25)^2}{\eta(2\tau+1)^2 \eta(25\tau+10)}\\
		&= \frac{\eta(\tau)\eta(50\tau)^2}{\eta(2\tau)^2 \eta(25\tau)}\\
		&= q^3 \Psi_{25}(\tau),
	\end{align*}
	as claimed.
\end{proof}

\subsubsection{Series $\Phi_{25}$ and $\Psi_{25}$}

The proof of the internal congruences for $\Phi_{25}$ in Theorem~\ref{th:ph25-mod5} is built on the polynomial representations in \eqref{eq:hL25-def-1}:
\begin{align*}
	\hL_{25}(\hxi,\alpha) = \sum_{n\ge 0} \ph_{25} \big(5^{\alpha} n\big) q^n,
\end{align*}
whereby each $\hL_{25}(\hxi,\alpha)\in \mathbb{Z}[\hxi]$. This tells us that
\begin{align}\label{eq:hL25-U5-relation}
	\hL_{25}(\hxi,\alpha) = U_5^{(\alpha)}(\Phi_{25}).
\end{align}
Once again, we replace $\hxi$ with a generic indeterminate $x$ and lift them to a family of polynomials $\hL_{25}(x,\alpha)\in \mathbb{Z}[x]$. According to \eqref{eq:hZ-ineq},
\begin{align}
	\hL_{25}(x,\alpha+1) - \hL_{25}(x,\alpha) \equiv 0 \pmod{5^{\alpha}}.
\end{align}
As a consequence, Theorem~\ref{th:ps25-mod5} follows from the next relation.

\begin{theorem}\label{th:hL25-hzeta}
	For any $\alpha\ge 1$,
	\begin{align}
		\hL_{25}(\hzeta,\alpha) = \sum_{n\ge 0} \ps_{25} \big(5^{\alpha} n + 5^{\alpha} -3\big) q^{n+1}.
	\end{align}
\end{theorem}

\begin{proof}
	Noting the fact that the operations of $\hV$ and $U_5$ are interchangeable, we have
	\begin{align*}
		\sum_{n\ge 0} \ps_{25} \big(5^{\alpha} n + 5^{\alpha} -3\big) q^{n+1} &=U_{5}^{(\alpha)}( q^3\Psi_{25} ) \overset{\eqref{eq:hV-hPh}}{=} U_{5}^{(\alpha)}\big( \Phi_{25}(\hV\tau) \big)\\
		&= U_{5}^{(\alpha)}( \Phi_{25} )(\hV\tau) \overset{\eqref{eq:hL25-U5-relation}}{=} \hL_{25}(\hxi,\alpha)(\hV\tau)\\
		&= \hL_{25}(\hxi(\hV\tau),\alpha)\overset{\eqref{eq:hV-hxi}}{=} \hL_{25}(\hzeta,\alpha),
	\end{align*}
	as desired.
\end{proof}

\subsubsection{Series $\Phi$ and $\Psi$}

To establish the internal congruence family for $\Phi$ in Theorem~\ref{th:ph-mod5}, we have constructed the following polynomial representations in \eqref{eq:hL-def-1}:
\begin{align*}
	\hL(\hxi,2\alpha-1) &= \varphi(-q^5)\sum_{n\ge 0} \ph \big(5^{2\alpha-1} n\big)q^n,\\
	\hL(\hxi,2\alpha) &= \varphi(-q)\sum_{n\ge 0} \ph \big(5^{2\alpha} n\big) q^n,
\end{align*}
whereby each $\hL(\hxi,\alpha) \in \mathbb{Z}[\hxi]$. In particular, we know from \eqref{eq:hL-U5-even} and \eqref{eq:hL-U5-odd} that
\begin{align}
	\hL(\hxi,2\alpha) &= U_5\big(\hL(\hxi,2\alpha-1)\big),\label{eq:hL-U5-even-1}\\
	\hL(\hxi,2\alpha+1) &= U_5\big(\Phi_{25}\cdot \hL(\hxi,2\alpha)\big).\label{eq:hL-U5-odd-1}
\end{align}
We again lift them to a family of polynomials $\hL(x,\alpha)\in \mathbb{Z}[x]$ by replacing $\hxi$ with an indeterminate $x$. In light of \eqref{eq:hN-even-ineq},
\begin{align}
	\hL(x,2\alpha+2)- \hL(x,2\alpha) &\equiv 0 \pmod{5^{\alpha}}.
\end{align}
Therefore, the congruences in Theorem~\ref{th:ps-mod5} can be confirmed by the following result.

\begin{theorem}
	For any $\alpha\ge 1$,
	\begin{align}
		\hL(\hzeta,2\alpha-1) &= \psi(q^5)\sum_{n\ge 0} \ps \!\left(5^{2\alpha-1} n + \frac{3\cdot 5^{2\alpha-1}+1}{8}\right)q^{n+1},\label{eq:hL-hzeta-ps-odd}\\
		\hL(\hzeta,2\alpha) &= \psi(q)\sum_{n\ge 0} \ps \!\left(5^{2\alpha} n + \frac{7\cdot 5^{2\alpha}+1}{8}\right) q^{n+1}.\label{eq:hL-hzeta-ps-even}
	\end{align}
\end{theorem}

\begin{proof}
	We begin with the $\alpha=1$ case of \eqref{eq:hL-hzeta-ps-odd}:
	\begin{align*}
		\psi(q^5) \sum_{n\ge 0} \ps(5n+2)q^{n+1} = U_5(q^3\Psi_{25}) = \hL_{25}(\hzeta,1) = \hL(\hzeta,1),
	\end{align*}
	where we have recalled the proof of Theorem~\ref{th:hL25-hzeta} together with the fact that
	\begin{align*}
		\hL(x,1) = \hL_{25}(x,1),
	\end{align*}
	which follows from
	\begin{align*}
		\hL(\hxi,1) = \varphi(-q^5)\sum_{n\ge 0} \ph (5n)q^n = U_5(\Phi_{25}) = \sum_{n\ge 0} \ph_{25} (5n) q^n = \hL_{25}(\hxi,1).
	\end{align*}
	
	Now we assume \eqref{eq:hL-hzeta-ps-odd} for a certain $\alpha\ge 1$. Then
	\begin{align*}
		\hL(\hzeta,2\alpha) &= \hL(\hxi,2\alpha)(\hV\tau) \overset{\eqref{eq:hL-U5-even-1}}{=} U_5\big(\hL(\hxi,2\alpha-1)\big)(\hV\tau)\\
		&= U_5\big(\hL(\hxi(\hV\tau),2\alpha-1)\big) \overset{\eqref{eq:hV-hxi}}{=} U_5\big(\hL(\hzeta,2\alpha-1)\big).
	\end{align*}
	Hence,
	\begin{align*}
		\hL(\hzeta,2\alpha) &= U_5\left(\psi(q^5)\sum_{n\ge 0} \ps \!\left(5^{2\alpha-1} n + \frac{3\cdot 5^{2\alpha-1}+1}{8}\right)q^{n+1}\right)\\
		&= \psi(q)\sum_{n\ge 0} \ps \!\left(5^{2\alpha} n + \frac{7\cdot 5^{2\alpha}+1}{8}\right) q^{n+1},
	\end{align*}
	which gives \eqref{eq:hL-hzeta-ps-even} for $\alpha$.
	
	Finally,
	\begin{align*}
		\hL(\hzeta,2\alpha+1) &= \hL(\hxi,2\alpha+1)(\hV\tau) \overset{\eqref{eq:hL-U5-odd-1}}{=} U_5\big(\Phi_{25}\cdot \hL(\hxi,2\alpha)\big)(\hV\tau)\\
		&= U_5\big(\Phi_{25}(\hV\tau)\cdot \hL(\hxi(\hV\tau),2\alpha)\big) \overunderset{\eqref{eq:hV-hxi}}{\eqref{eq:hV-hPh}}{=} U_5\big(q^3\Psi_{25}\cdot \hL(\hzeta,2\alpha)\big).
	\end{align*}
	We conclude that
	\begin{align*}
		\hL(\hzeta,2\alpha+1) &= U_5\left(q^3\frac{\psi(q^{25})}{\psi(q)}\cdot \psi(q)\sum_{n\ge 0} \ps \!\left(5^{2\alpha} n + \frac{7\cdot 5^{2\alpha}+1}{8}\right) q^{n+1}\right)\\
		&= \psi(q^5)\sum_{n\ge 0} \ps \!\left(5^{2\alpha+1} n + \frac{3\cdot 5^{2\alpha+1}+1}{8}\right)q^{n+1},
	\end{align*}
	which is exactly \eqref{eq:hL-hzeta-ps-odd} with $\alpha$ replaced by $\alpha+1$.
\end{proof}

\section{Algebraic considerations}\label{algconsiderations}

In the results above we have studied several different examples of multiplicities between internal congruences.  In each case, we constructed a map---in these cases, an Atkin--Lehner involution---between the associated generating functions manifesting the congruences of interest.  This map acts as an isomorphism between the associated function spaces, or in our situation, the spaces of modular functions which live at a specific cusp of the underlying modular curve.

A natural question arises: what functions are \textit{invariant} with respect to our involutions?  What follows is an extremely elementary exercise in algebra; nevertheless, we expect that the implications will strike the algebraist with curiosity.

For the moment we consider just $\Gamma_0(6)$, especially Section~\ref{sec:ps9-mod3}. One can prove with a straightforward cusp analysis that the Hauptmoduln $\xi$ and $\zeta$ in \eqref{eq:Hauptmodul-xi} and \eqref{eq:Hauptmodul-zeta} satisfy
\begin{align*}
	\zeta = \frac{\xi-1}{9\xi-1}.
\end{align*}  A natural choice of a function which is invariant under the operation of $V$ given in \eqref{eq:V-def} would be either $\xi+\zeta$ or $\xi\zeta$.  In fact, we have
\begin{align}
	\xi\zeta = \frac{\xi^2-\xi}{9\xi-1},\label{eq:xi-zeta-1}
\end{align}
and
\begin{align}
	\xi + \zeta &= \frac{9\xi^2-1}{9\xi-1} = 9\xi\zeta + 1.\label{eq:xi-zeta-2}
\end{align}
To avoid annoying negative powers of $3$, it makes the most sense to work with $\xi\zeta$.  Therefore, let us define
\begin{align*}
	t = t(\tau) := \xi\zeta.
\end{align*}  We immediately have
\begin{align*}
	t(V\tau) = t.
\end{align*}

This is a quite natural choice of invariant function under the operation of $V$.  Indeed, we have seen that $V$ maps the Hauptmoduln at $[0]_6$ to those at $[1/2]_6$, and vice versa.  It can be shown in similar fashion that $V$ interchanges the Hauptmoduln at $[\infty]_6$ with those at $[1/3]_6$.

As such, any nontrivial function which is invariant under the action of $V$ must have poles of matching order at two cusps --- either at $[0]_6$ and $[1/2]_6$, or at $[\infty]_6$ and $[1/3]_6$.

Let us first consider a function $h$ with poles of matching order at $[0]_6$ and $[1/2]_6$.  Such a function $h$ has the form
\begin{align*}
h = f(\xi) + g(\zeta),
\end{align*} for some $f,g\in\mathbb{C}[X]$.  But we also know by hypothesis that $h$ must be invariant under application of $V$:
\begin{align*}
h(V\tau) = f(\zeta) + g(\xi) = f(\xi) + g(\zeta).
\end{align*}  Thus the polynomial sum
\begin{align*}
f(X)+g(Y)\in\mathbb{C}[X,Y]
\end{align*} must be symmetric in $X$ and $Y$, and is therefore a polynomial in the elementary symmetric polynomials $XY$, $X+Y$.

Thus, $h$ is a polynomial in $\xi\zeta$ and $\xi+\zeta$.  But we know from (\ref{eq:xi-zeta-2}) that $\xi+\zeta$ can be written as an expression in $\xi\zeta = t$, so that $h$ must be a polynomial in $t$.

Consider next a function with matching poles at $[\infty]_6$ and $[1/3]_6$.  Natural Hauptmoduln to consider are simply
\begin{align*}
	\mu := \frac{1}{\xi},\qquad\qquad \pi := \frac{1}{\zeta}.
\end{align*}
Here $\mu$ lives at the cusp $[1/3]_6$, and $\pi$ lives at $[\infty]_6$.  We see then that any modular function which has poles of matching order at these poles must be a polynomial in $1/\xi$ and $1/\zeta$.  Once again, the theory of symmetric functions demands that any function living at these poles which is invariant under $V$ must be a polynomial in
\begin{align*}
\frac{1}{\xi\zeta}=\frac{1}{t},
\end{align*}
and
\begin{align*}
	\frac{1}{\xi}+\frac{1}{\zeta} = \frac{\xi+\zeta}{\xi\zeta} = \frac{9t+1}{t}.
\end{align*}
We have therefore proved the following:
\begin{theorem}\label{tinvtheorem}
Any modular function over $\Gamma_0(6)$ which is invariant under $V$ must be a rational polynomial in $t$.
\end{theorem}

The implications of this result become apparent in light of \eqref{eq:xi-zeta-1}, wherein we have
\begin{align*}
	t &=\frac{\xi(\xi-1)}{9\xi-1} = \frac{\xi^2-\xi}{9\xi-1}.
\end{align*}  We see that $X=\xi$ is a solution to the following \emph{quadratic} equation with coefficients in $\mathbb{Z}[t]$:
\begin{align}\label{mode3}
	X^2-(9t+1)X+t=0.
\end{align}
One can almost immediately verify that $X=\zeta$ is the second solution.

There are two noteworthy implications of this.

First, we have
\begin{align}\label{mode3a}
	X^2=(9t+1)X-t,
\end{align}
for $X=\xi$ and $\zeta$, whence we can express any integer polynomial in $\xi$ as a member of a rank two $\mathbb{Z}[t]$-module:
\begin{align*}
	\mathbb{Z}[\xi] = \mathbb{Z}[t] + \xi \mathbb{Z}[t].
\end{align*}  Moreover, in changing to this expression, our $3$-adic valuation is retained.

For example, consider $L_9(\xi,2) - L_9(\xi,1)$ as in \eqref{eq:ph9-i=1}:
\begin{align*}
	L_9(\xi,2) - L_9(\xi,1) = 9 \xi - 36 \xi^2 + 27 \xi^3.
\end{align*}
Accounting for \eqref{mode3}, we can reduce this expression to
\begin{align*}
	L_9(\xi,2) - L_9(\xi,1) = \mathcal{L}^{(0)}_9(t,1)+\xi\cdot\mathcal{L}^{(1)}_9(t,1)
\end{align*} with
\begin{align*}
	\mathcal{L}^{(0)}_9(t,1) &= 9 t (1 - 27 t),\\
	\mathcal{L}^{(1)}_9(t,1) &= 27 t (5 + 81 t) x.
\end{align*}
Since $t$ is invariant with respect to our involution $V$, we immediately have
\begin{align*}
	L_9(\zeta,2) - L_9(\zeta,1) &= \mathcal{L}^{(0)}_9(t,1)+\zeta\cdot\mathcal{L}^{(1)}_9(t,1).
\end{align*}
That is to say, our polynomial expressions are wholly invariant, except for the presence of a single factor of $\xi$ or $\zeta$, which acts as a \emph{switch} between the two different cases.  We have a corresponding $\mathcal{L}^{(i)}_9(t,\alpha)$ for every $\alpha\ge 1$ and $i\in\{0,1\}$, and
\begin{align*}
	\mathcal{L}^{(i)}_9(t,\alpha)\equiv 0\pmod{3^{2\alpha}}.
\end{align*}

We see here that each of these different internal congruence families correspond to a pair of $3$-adic 0 sequences
\begin{align}
\mathcal{L}^{(i)}:=\big\{ \mathcal{L}^{(i)}_9(t,\alpha) \big\}_{\alpha\ge 1}\label{mathcalLdefn}
\end{align}
for $i\in\{0,1\}$ within the function field $\mathbb{Q}(t)$.

The second implication is more abstract, but at least as interesting.  Here we have an internal congruence family which has multiplicity, i.e., it manifests itself in the coefficients of two different modular forms, via \eqref{eq:ph9-mod3} and \eqref{eq:ps9-mod3}.  A natural question is whether this family manifests elsewhere: does this family exhibit \textit{additional} multiplicity?

Consider that the functions $\xi^{\pm 1}$, $\zeta^{\pm 1}$ can be used to define all of the modular functions over $\Gamma_0(6)$, since these account for Hauptmoduln at each cusp of $\mathrm{X}_0(6)$.  We see that $V$ interchanges $\xi$ and $\zeta$, and that by Theorem \ref{tinvtheorem}, \textit{all} functions which are invariant under $V$ are rational polynomials in $t$.

Finally, we have $\xi$ and $\zeta$ as two distinct roots of a polynomial with coefficients in $\mathbb{Z}[t]$.  Thus we can define the quadratic field extension $\mathbb{K} := \mathbb{Q}(t)(\xi)$, with \eqref{mode3} as the associated minimal polynomial.

Our involution then defines an automorphism on $\mathbb{K}$ which fixes $\mathbb{Q}(t)$ and permutes the roots of \eqref{mode3}:
\begin{align*}
	\begin{array}{ccccc}
		\sigma & : & \mathbb{K} & \rightarrow & \mathbb{K},\\
		& & t & \mapsto & t,\\
		& & \xi & \mapsto & \zeta.
	\end{array}
\end{align*}

We note that $\sigma^2$ is the identity mapping, and that $\mathbb{K}$ is a quadratic extension of the invariant field $\mathbb{Q}(t)$.  The only automorphic maps which fix $\mathbb{Q}(t)$ are the identity mapping and $\sigma$ itself.  This suggests that \textit{no other} internal congruence families exist over $\Gamma_0(6)$ which are equivalent to \eqref{eq:ph9-mod3} and \eqref{eq:ps9-mod3}.

Were \eqref{mode3} an equation of higher degree, say with a distinct third root $\omega$, we would expect that with appropriate adjustments on the definition of $t$, we could construct a more complex field extension and corresponding automorphism group containing additional mappings, whereupon the sequences $\mathcal{L}^{(i)}$ in (\ref{mathcalLdefn}) above would attach to $\omega$, and \eqref{eq:ph9-mod3} and \eqref{eq:ps9-mod3} would be equivalent to a third congruence family of interest.  Here we see the potential of this approach in predicting when congruences do (or do not) exist.

We may carry out the same analysis to $\Gamma_0(10)$. This time the two Hauptmoduln $\hxi$ and $\hzeta$ are defined in \eqref{eq:Hauptmodul-hxi} and \eqref{eq:Hauptmodul-hzeta}, and they jointly satisfy the relation
\begin{align*}
	\hzeta = \frac{\hxi - 1}{5\hxi - 1}.
\end{align*}
Then we choose
\begin{align*}
	\htt = \htt(\tau) := \hxi\hzeta,
\end{align*}
which is invariant under the operation of $\hV$ given in \eqref{eq:hV-def}. Furthermore, the two solutions to the quadratic equation
\begin{align}
	X^2-(5\htt+1)X+\htt=0
\end{align}
are $X=\hxi$ and $X=\hzeta$, and similar arguments apply.

\subsection*{Acknowledgements}

Shane Chern was funded by the Austrian Science Fund (FWF) SFB Project (No.~10.55776/F1002, ``Discrete Random Structures: Enumeration and Scaling Limits''). Nicolas Allen Smoot was funded in whole by the Austrian Science Fund (FWF) Principal Investigator Project (No.~10.55776/ PAT2615425, ``Multiplicities Between Modular Congruences"). Dazhao Tang was funded in part by the National Natural Science Foundation of China (No.~12201093), and by the Science and Technology Research Program of Chongqing Municipal Education Commission (No.~KJQN202500501).

For open access purposes, the authors have applied a CC BY public copyright license to any author-accepted manuscript version arising from this submission.

In particular, Shane Chern and Nicolas Allen Smoot would like to thank the Austrian Government and People for their generous support.

\bibliographystyle{amsplain}

\end{document}